\documentclass[12pt]{article}
\pdfoutput=1
\usepackage[utf8]{inputenc}
\usepackage{amsfonts,amsmath,amsthm,amssymb}
\usepackage{nicefrac}
\usepackage{enumitem}
\usepackage[colorlinks,allcolors=blue]{hyperref}
\usepackage{url}
\usepackage[longnamesfirst, authoryear]{natbib}
\usepackage[title]{appendix}
\usepackage{breakcites}
\usepackage[ruled,vlined,noresetcount]{algorithm2e}
\SetKw{kwParams}{Parameters:}
\SetKw{kwInit}{Initializaton:}
\usepackage[margin=1in]{geometry}
\usepackage{graphicx}
\usepackage{float}

\newtheorem{theorem}{Theorem}

\newtheorem{lemma}{Lemma}

\newtheorem{proposition}{Proposition}

\newtheorem{corollary}{Corollary}

\theoremstyle{remark}
\newtheorem{remark}{Remark}

\newtheorem{definition}{Definition}

\newtheorem{example}{Example}

\let\Pr\relax
\DeclareMathOperator\Pr{\mathbb{P}}
\DeclareMathOperator\Qr{\mathbb{Q}}
\DeclareMathOperator\E{\mathbb{E}}

\DeclareMathOperator\Ber{Bernoulli}

\newcommand{\KL}[2]{ { D \left({#1} \;\middle\Vert\; {#2}\right) } }
\newcommand{\KLr}[2]{ { D \left({#1} \;\middle\Vert\; {#2}\right) } }
\newcommand{\KLd}[2]{ { \dot{D} \left({#1} \;\middle\Vert\; {#2}\right) } }
\newcommand{\KLrd}[2]{ { \dot{D} \left({#1} \;\middle\Vert\; {#2}\right) } }
\newcommand{\KLb}[2]{ { D_2 \left({#1} \;\middle\Vert\; {#2}\right) } }

\newcommand{\btheta}{{\pmb{\theta}}}

\newcommand{\blambda}{{\pmb{\lambda}}}
\newcommand{\bphi}{{\pmb{\phi}}}
\newcommand{\Prtheta}{\Pr_\btheta^\bphi}
\newcommand{\Etheta}{\E_\btheta^\bphi}
\newcommand{\Prlambda}{\Pr_\blambda^\bphi}
\newcommand{\Elambda}{\E_\blambda^\bphi}

\newcommand{\calE}{{\mathcal E}}
\newcommand{\calF}{{\mathcal F}}

\newcommand{\calM}{{\mathcal M}}

\newcommand{\calX}{{\mathcal X}}

\newcommand{\Int}[1]{ { {\mathbb Z}^{#1} } }
\newcommand{\Intp}[1]{ { {\mathbb Z_{\ge 0}}^{#1} } }
\newcommand{\Intpp}[1]{ { {\mathbb Z_{> 0}}^{#1} } }
\newcommand{\Real}[1]{ { {\mathbb R}^{#1} } }
\newcommand{\Realp}[1]{ { {\mathbb R}_{\ge 0}^{#1} } }

\newcommand{\dss}{\displaystyle}

\newcommand{\inv}{^{-1}}

\DeclareMathOperator*{\argmin}{arg\,min}

\title{\textbf{Finite-Time Analysis of Round-Robin Kullback-Leibler Upper Confidence Bounds for Optimal Adaptive Allocation with Multiple Plays and Markovian Rewards}}
\author{
Vrettos Moulos
\thanks{Supported in part by the NSF grant CCF-1816861.}
\\
University of California Berkeley \\
\href{mailto:vrettos@berkeley.edu}{vrettos@berkeley.edu}
}
\date{}

\begin{document}

\maketitle

\begin{abstract}
We study an extension of the classic stochastic multi-armed bandit problem which involves multiple plays and Markovian rewards in the rested bandits setting.
In order to tackle this problem we consider an adaptive allocation rule which at each stage combines the information from the sample means of all the arms,
with the Kullback-Leibler upper confidence bound of a single arm which is selected in round-robin way.
For rewards generated from a one-parameter exponential family of Markov chains, we provide a finite-time upper bound for the regret incurred from this adaptive allocation rule, which reveals the logarithmic dependence of the regret on the time horizon, and which is asymptotically optimal.
For our analysis we devise several concentration results for Markov chains,
including a maximal inequality for Markov chains,
that may be of interest in their own right.
As a byproduct of our analysis we also establish asymptotically optimal, finite-time guarantees for the case of multiple plays, and i.i.d. rewards drawn from a one-parameter exponential family of probability densities.
Additionally, we provide simulation results that illustrate that calculating Kullback-Leibler upper confidence bounds in a round-robin way, is significantly more efficient than calculating them for every arm at each round, and that the expected regrets of those two approaches behave similarly.
\end{abstract}

\section{Introduction}

In this paper we study a generalization of the stochastic multi-armed bandit problem,
where there are $K$ independent arms, and each arm $a \in [K] = \{1, \ldots, K\}$ is associated with a parameter $\theta_a \in \Real{}$,
and modeled as a discrete time stochastic process governed by the probability law $\Pr_{\theta_a}$.
A time horizon $T$ is prescribed, and at each round $t \in [T] = \{1, \ldots, T\}$ we select $M$ arms, where $1 \le M \le K$, without any prior knowledge of the statistics of the underlying stochastic processes.
The $M$ stochastic processes that correspond to the selected arms evolve by one time step, and we observe this evolution through a reward function, while the stochastic processes for the rest of the arms stay frozen, i.e. we consider the rested bandits setting.
Our goal is to select arms in such a way so as to make the cumulative reward over the whole time horizon $T$ as large as possible.
For this task we are faced with an exploitation versus exploration dilemma.
At each round we need to decide whether we are going to exploit the best $M$ arms according to the information that we have gathered so far, or we are going to explore some other arms which do not seem to be so rewarding, just in case that the rewards we have observed so far deviate significantly from the expected rewards.
The answer to this dilemma is usually coming by calculating indices for the arms and ranking them according to those indices,
which should incorporate both information on how good an arm seems to be as well as on how many times it has been played so far.
Here we take an alternative approach where instead of calculating the indices for all the arms at each round, we just calculate the index for a single arm in a round-robin way.

\subsection{Contributions}

\begin{enumerate}
    \item
We first consider the case that the $K$ stochastic processes are irreducible Markov chains, coming from a one-parameter exponential family of Markov chains.
The objective is to play as much as possible the $M$ arms with the largest stationary means, although we have no prior information about the statistics of the $K$ Markov chains.
The difference of the best possible expected rewards  coming from those $M$ best arms and the expected reward coming from the arms that we played is the regret that we incur.
To minimize the regret we consider an index based adaptive allocation rule,~\autoref{alg:KL-UCB},
which is based on sample means and Kullback-Leibler upper confidence bounds for the stationary expected rewards using the Kullback-Leibler divergence rate.
We provide a finite-time analysis,~\autoref{thm:main}, for this KL-UCB adaptive allocation rule which shows that the regret depends logarithmically on the time horizon $T$, 
and matches exactly the asymptotic lower bound,~\autoref{cor:optimal}.

    \item In order to make the finite-time guarantee possible we devise several deviation lemmata for Markov chains. An exponential martingale for Markov chains is proven,~\autoref{lem:martingale},
    which leads to a maximal inequality for Markov chains,~\autoref{lem:maximal}. In the literature there are several approaches that use martingale techniques 
    either to derive Hoeffding inequalities for Markov chains~\cite{Glynn-Ormoneir-02,Moulos-Hoeffding-20},
    or more generally to study concentration of measure for  
    Markov chains~\cite{Marton-96-a,Marton-96-b,Marton-98,Samson-00,Marton-03,Chazottes-Collet-Kulske-Redig-07,Kontorovich-Ramanan-08,Paulin15}.
    Nonetheless, they're all based either on Dynkin's martingale or on Doob's martingale, combined with coupling ideas,
    and there is no evidence that they can lead to maximal inequalities. Moreover, a Chernoff bound for Markov chains is devised,~\autoref{lem:Chernoff}, and its relation with the work of~\cite{Moulos-Ananth-19} is discussed in~\autoref{rmk:Chernoff-comparison}.
    
    \item We then consider the case that the $K$ stochastic processes are i.i.d. processes, each corresponding to a density coming from a one-parameter exponential family of densities. We establish,~\autoref{thm:main-IID}, that~\autoref{alg:KL-UCB} still enjoys the same finite-time regret guarantees, which are asymptotically optimal. 
    The case where~\autoref{thm:main-IID} follows directly from~\autoref{thm:main} is discussed in~\autoref{rmk:IID-Markov}.
    The setting of single plays is studied in~\cite{Cappe-Garivier-Maillard-Munos-Stoltz-13}, but with a much more computationally intense adaptive allocation rule.
    
    \item In~\autoref{sec:simulations} we provide simulation results illustrating the fact that round-robin KL-UCB adaptive allocation rules are much more computationally efficient than
    KL-UCB adaptive allocation rules, and similarly round-robin UCB adaptive allocation rules are more computationally efficient than
    UCB adaptive allocation rules, while the expected regrets, in each family of algorithms, behave in a similar way.
    This brings to light round-robin schemes as an appealing practical alternative to the mainstream schemes that calculate indices for all the arms at each round.
    
\end{enumerate}

\subsection{Motivation}

Multi-armed bandits provide a simple abstract statistical model that can be applied to study real world problems such as clinical trials, ad placement,
gambling, adaptive routing, resource allocation in computer systems etc.
We refer the interested reader to the survey of~\cite{BC12} for more context, and to the recent books of~\cite{Lat-Szep-19,Slivkins-19}. The need for multiple plays can be understood in the setting of resource allocation. Scheduling jobs to a single CPU is an instance of the multi-armed bandit problem with a single play at each round, where the arms correspond to the jobs.
If there are multiple CPUs we get an instance of the multi-armed bandit problem with multiple plays.
The need of a richer model which allows the presence of Markovian dependence is illustrated in the context of gambling, where the arms correspond to slot-machines.
It is reasonable to try to model the assertion that if a slot-machine
produced a high reward the $n$-th time played, then it is very likely that it will produce a much lower reward the $(n+1)$-th time played, simply because the casino may decide to change the reward distribution to a much stingier one if a big reward was just produced. This assertion requires, the reward distributions to depend on the previous outcome, which is precisely captured by the Markovian reward model. 
Moreover, we anticipate this to be an important problem attempting to bridge classical stochastic bandits, controlled Markov chains (MDPs), and non-stationary bandits.

\subsection{Related Work}

The cornerstone of the multi-armed bandits literature is the pioneering work of~\cite{Lai-Robbins-85}, which studies the problem for the case of i.i.d. rewards and single plays.~\cite{Lai-Robbins-85} introduce the change of measure argument to derive a lower bound for the problem, as well as round robin adaptive allocation rules based on upper confidence bounds which are proven to be asymptotically optimal.~\cite{Ananth-Varaiya-Walrand-I-87} extend the results of~\cite{Lai-Robbins-85} to the case of i.i.d. rewards and multiple plays, while~\cite{Agrawal-95} considers index based allocation rules which are only based on sample means and are computationally simpler, although they may not be asymptotically optimal. 
The work of~\cite{Agrawal-95} inspired the
first finite-time analysis for the adaptive allocation rule called UCB
by~\cite{Auer-Cesa-Fischer-02}, which is though asymptotically suboptimal.
The works of~\cite{Cappe-Garivier-Maillard-Munos-Stoltz-13,Garivier-Cappe-11, Maillard-Munos-Stoltz-11} bridge this gap by providing the KL-UCB adaptive allocation rule, with finite-time guarantees which are asymptotically optimal.
Additionally,~\cite{KHN-15} study a Thompson sampling algorithm for multiple plays and binary rewards, and they establish a  finite-time analysis which is asymptotically optimal. Here we close the problem of multiple plays and rewards coming from an exponential family of probability densities by showing finite-time guarantees which are asymptotically optimal, via adaptive allocation rules which are much more efficiently computable than their precursors.

The study of Markovian rewards and multiple plays in the rested setting, is initiated in the work of~\cite{Ananth-Varaiya-Walrand-II-87}.
They report an asymptotic lower bound, as well as a round robin upper confidence bound adaptive allocation rule which is proven to be asymptotically optimal.
However, it is unclear if the statistics that they use in order to derive the upper confidence bounds, in their Theorem 4.1, can be recursively computed, and the practical applicability of their results is therefore questionable.
In addition, they don't provide any finite-time analysis, and they use a different type of assumption on their one-parameter family of Markov chains. In particular, they assume that their one-parameter family of transition probability matrices is log-concave in the parameter,
equation (4.1) in~\cite{Ananth-Varaiya-Walrand-II-87}, while we assume that it is a one-parameter exponential family of transition probability matrices.
\cite{Tekin-Liu-10,Tekin-Liu-12} extend the UCB adaptive allocation rule of~\cite{Auer-Cesa-Fischer-02}, to the case of Markovian rewards and multiple plays. They provide a finite-time analysis, but their regret bounds are suboptimal. Moreover they impose a different type of assumption on their configuration of Markov chains. They assume that the transition probability matrices are reversible, so that they can apply Hoeffding bounds for Markov chains from~\cite{Gillman93, Lezaud98}.
In a recent work~\cite{Moulos-Hoeffding-20} developed a Hoeffding bound for Markov chains, which does not assume any conditions other than irreducibility, and using this he extended the analysis of UCB to an even broader class of Markov chains.
One of our main contributions is to bridge this gap and provide a KL-UCB adaptive allocation rule, with a finite-time guarantee which is asymptotically optimal. In a different line of work~\cite{ORAM-12, Tekin-Liu-12} consider the restless bandits Markovian reward model, in which the state of each arm evolves according to a Markov chain independently of the player's action. Thus in the restless setting the state that we next observe is now dependent on the amount
of time that elapses between two plays of the same arm.
\section{Problem Formulation}

\subsection{One-Parameter Family of Markov Chains}\label{sec:general-family}

We consider a one-parameter family of irreducible Markov chains on
a finite state space $S$.
Each member of the family is indexed by a parameter
$\theta \in \Real{}$,
and is characterized by an initial distribution
$q_\theta = [q_\theta (x)]_{x \in S}$,
and an irreducible transition probability matrix
$P_\theta = [P_\theta (x,y)]_{x, y \in S}$,
which give rise to a probability law $\Pr_\theta$.
There are $K \ge 2$ arms,
with overall parameter configuration
$\btheta = (\theta_1, \ldots, \theta_K) \in \Real{}^K$,
and each arm $a \in [K] = \{1, \ldots, K\}$ evolves internally
as the Markov chain with parameter $\theta_a$ which we denote by
$\{X_n^a\}_{n \in \Intp{}}$.
There is a common noncostant real-valued reward function on the state space
$f : S \to \Real{}$, and successive plays of arm $a$
result in observing samples from the stochastic process
$\{Y_n^a\}_{n \in \Intp{}}$, where $Y_n^a = f (X_n^a)$. 
In other words, the distribution of the rewards coming from arm $a$ is a function of the Markov chain with parameter $\theta_a$,
and thus it can have more complicated dependencies.
As a special case, if we pick the reward function $f$ to be injective,
then the distribution of the rewards is Markovian.

For $\theta \in \Real{}$, due to irreducibility,
there exists a unique stationary distribution for the transition probability matrix $P_\theta$ which we denote with
$\pi_\theta = [\pi_\theta (x)]_{x \in S}$.
Furthermore, let $\mu (\theta) = \sum_{x \in S} f (x) \pi_\theta (x)$ be the stationary mean reward corresponding to the Markov chain parametrized by $\theta$. Without loss of generality we may assume that the $K$ arms are ordered so that,
\[
\mu (\theta_1) \ge \ldots \ge \mu (\theta_N) > 
\mu (\theta_{N+1}) \ldots = \mu (\theta_M) = \ldots = \mu (\theta_L) >
\mu (\theta_{L+1}) \ge \ldots \ge \mu (\theta_K),
\]
for some $N \in \{0, \ldots, M-1\}$ and $L \in \{M, \ldots, K\}$,
where $N = 0$ means that $\mu (\theta_1) = \ldots = \mu (\theta_M),~
L = K$ means that $\mu (\theta_M) = \ldots = \mu (\theta_K)$,
and we set $\mu (\theta_0) = \infty$ and $\mu (\theta_{K+1}) = - \infty$.

\subsection{Regret Minimization}

We fix a time horizon $T$, and at each round $t \in [T] = \{1, \ldots, T\}$ we play a set $\phi_t$ of $M$ distinct arms, where $1 \le M \le K$ is the same through out the rounds,
and we observe rewards $\{Z_t^a\}_{a \in [K]}$ given by,
\[
Z_t^a =
\begin{cases}
Y_{N_a (t)}^a, & \text{if}~ a \in \phi_t \\
0, & \text{if}~ a \not\in \phi_t,
\end{cases}
\]
where $N^a (t) = \sum_{s=1}^t I \{a \in \phi_s\}$ is the number of times we played arm $a$ up to time $t$.
Using the stopping times $\tau_n^a = \inf \{t \ge 1 : N^a (t) = n\}$,
we can also reconstruct the $\{Y_n^a\}_{n \in \Intpp{}}$ process,
from the observed $\{Z_t^a\}_{t \in \Intpp{}}$ process, via the identity $Y_n^a = Z_{\tau_n^a}^a$.
Our play $\phi_t$ is based on the information that we have accumulated so far.
In other words, the event $\{\phi_t = A\}$, for $A \subseteq [K]$ with $|A| = M$, belongs to the $\sigma$-field generated by
$\phi_1, \{Z_1^a\}_{a \in [K]}, \ldots, \phi_{t-1}, \{Z_{t-1}^a\}_{a \in [K]}$. We call the sequence $\bphi = \{\phi_t\}_{t \in \Int{}_{>0}}$ of our plays an \emph{adaptive allocation rule}. Our goal is to come up with an adaptive allocation rule $\bphi$,
that achieves the greatest possible expected value for the sum of the rewards,
\[
S_T
= \sum_{t=1}^T \sum_{a \in [K]} Z_t^a
= \sum_{a \in [K]} \sum_{n=1}^{N^a (T)} Y_n^a,
\]
which is equivalent to minimizing the expected regret,
\begin{equation}\label{eqn:regret}
R_\btheta^\bphi (T) =
T \sum_{a=1}^M \mu (\theta_a) - \E_\btheta^\bphi [S_T].
\end{equation}

\subsection{Asymptotic Lower Bound}

A quantity that naturally arises in the study of regret minimization for Markovian bandits is the \emph{Kullback-Leibler divergence rate} between
two Markov chains, which is a generalization of the usual Kullback-Leibler divergence between two probability distributions.
We denote by $\KLr{\theta}{\lambda}$ the Kullback-Leibler divergence rate between
the Markov chain with parameter $\theta$ and the Markov chain with parameter $\lambda$, which is given by,
\begin{equation}\label{eqn:KLr}
\KLr{\theta}{\lambda} = 
\sum_{x, y \in S} \log \frac{P_\theta (x, y)}{P_\lambda (x, y)}
\pi_\theta (x) P_\theta (x, y),
\end{equation}
where we use the standard notational conventions $\log 0 = \infty, ~ \log \frac{\alpha}{0} = \infty ~ \text{if} ~ \alpha > 0$,
and $0 \log 0 = 0 \log \frac{0}{0} = 0$. Indeed note that, if
$P_\theta (x, \cdot) = p_\theta (\cdot)$ and $P_\lambda (x, \cdot) = p_\lambda (\cdot)$, for all $x \in S$, i.e. in the special case that the Markov chains correspond to IID processes, then
the Kullback-Leibler divergence rate $\KLr{\theta}{\lambda}$
is equal to the Kullback-Leibler divergence $\KL{p_\theta}{p_\lambda}$ between $p_\theta$ and $p_\lambda$,
\[
\KLr{\theta}{\lambda} =
\sum_{x, y \in S} \log \frac{p_\theta (y)}{p_\lambda (y)}
p_\theta (x) p_\theta (y) =
\sum_{y \in S} \log \frac{p_\theta (y)}{p_\theta (y)} p_\theta (y) =
\KL{p_\theta}{p_\lambda}.
\]

Under some regularity assumptions on the one-parameter family of Markov chains,~\cite{Ananth-Varaiya-Walrand-II-87} in their Theorem 3.1 are able to establish the following asymptotic lower bound on the expected regret for any adaptive allocation rule $\bphi$ which is uniformly good across all parameter configurations,
\begin{equation}\label{eqn:lower-bound}
\liminf_{T \to \infty} \frac{R_\btheta^\bphi (T)}{\log T}
\ge
\sum_{b=L+1}^K
\frac{\mu (\theta_M) -\mu (\theta_b)}
{\KLr{\theta_b}{\theta_M}}.
\end{equation}
A further discussion of this lower bound, as well as an alternative derivation can be found in~\autoref{a:lower-bound},

The main goal of this work is to derive a finite time analysis for an adaptive allocation rule which is based on Kullback-Leibler divergence rate indices, that is asymptotically optimal. We do so for the one-parameter exponential family of Markov chains, which forms a generalization of the classic one-parameter exponential family generated by a probability distribution with finite support.

\subsection{One-Parameter Exponential Family Of Markov Chains}\label{sec:exp-fam}

Let $S$ be a finite state space, $f : S \to \Real{}$ be a nonconstant reward function on the state space, and $P$ an irreducible transition probability matrix on $S$, with associated stationary distribution $\pi$.
$P$ will serve as the generator stochastic matrix of the family.
Let $\mu (0) = \sum_{x \in S} f (x) \pi (x)$ be the stationary mean of the Markov chain induced by $P$ when $f$ is applied.
By tilting exponentially the transitions of $P$ we are able to construct
new transition matrices that realize a whole range of stationary means around $\mu (0)$ and form the exponential family of stochastic matrices.
Let $\theta \in \Real{}$, and consider the matrix
$\Tilde{P}_\theta (x, y) = P (x, y) e^{\theta f (y)}$.
Denote by $\rho (\theta)$ its spectral radius. 
According to the Perron-Frobenius theory, 
see Theorem 8.4.4 in the book of~\cite{HJ13}, 
$\rho (\theta)$ is a simple eigenvalue of
$\Tilde{P}_\theta$, called the Perron-Frobenius eigenvalue, and we can associate to it unique left $u_\theta$ and right $v_\theta$ eigenvectors such that they are both positive, $\sum_{x \in S} u_\theta (x) = 1$ and $\sum_{x \in S} u_\theta (x) v_\theta (x) = 1$.
Using them we define the member of the exponential family which corresponds to the natural parameter $\theta$ as,
\begin{equation}\label{eqn:exp-fam-MC}
P_\theta (x, y) =
\frac{v_\theta (y)}{v_\theta (x)}
\exp\left\{\theta f (y) - \Lambda (\theta)\right\}
P (x, y),
\end{equation}
where $\Lambda (\theta) = \log \rho (\theta)$ is the log-Perron-Frobenius eigenvalue.
It can be easily seen that $P_\theta (x, y)$ is indeed a stochastic matrix, and its stationary distribution is given by $\pi_\theta (x) = u_\theta (x) v_\theta (x)$.
The initial distribution $q_\theta$ associated to the parameter $\theta$, can be any distribution on $S$, since the KL-UCB adaptive allocation rule that we devise, and its guarantees, will be valid no matter the initial distributions.

\begin{example}[Two-state chain]\label{e:example}
Let $S = \{0, 1\}$, and consider the transition probability matrix, $P$, representing two coin-flips, $\Ber (p)$ when we're in state $0$, and $\Ber (q)$ when we're in state $1$.
We require that $P$ is irreducible, so $p \in (0, 1]$ and $q \in [0, 1)$.
\[
P = 
\begin{bmatrix}
1-p & p \\
1-q & q
\end{bmatrix}
\]
The exponential family of transition probability matrices generated by $P$ and $f (x) = 2x-1$ is given by,
\[
P_\theta = 
\frac{1}{\rho (\theta)}
\begin{bmatrix}
(1-p)e^{-\theta} & \rho (\theta) - (1-p)e^{-\theta} \\
\rho (\theta) - q e^\theta & q e^\theta
\end{bmatrix},
\]
where,
\[
\rho (\theta) = \frac{(1-p)e^{-\theta} + qe^\theta + \sqrt{((1-p)e^{-\theta} - qe^\theta)^2 + 4p(1-q)}}{2}.
\]
In the special case that $p = q$, we get back the typical exponential family of $\Ber (p_\theta)$ coin-flips, with
\[
1 - p_\theta =
\frac{(1-p)e^{-\theta}}{(1-p)e^{-\theta} + pe^\theta}.
\]
\end{example}

Exponential families of Markov chains date back to the work of~\cite{Miller61}. 
For a short overview of one-parameter exponential families of Markov chains, as well as proofs of the following properties, we refer the reader to Section 2 in~\cite{Moulos-Ananth-19}.
The log-Perron-Frobenius eigenvalue $\Lambda (\theta)$ is a convex analytic function on the real numbers, and through its derivative, $\dot{\Lambda} (\theta)$, we obtain the stationary mean $\mu (\theta)$
of the Markov chain with transition matrix $P_\theta$ when $f$ is applied, i.e. $\dot{\Lambda} (\theta) = \mu (\theta) = \sum_{x \in S} f (x) \pi_\theta (x)$.
When $\Lambda (\theta)$ is not the linear function $\theta \mapsto \mu (0) \theta$, the log-Perron-Frobenius eigenvalue, $\Lambda (\theta)$, is strictly convex and thus its derivative $\dot{\Lambda} (\theta)$ is strictly increasing, and it forms a bijection
between the natural parameter space, $\Real{}$, and the mean parameter
space, $\calM = \dot{\Lambda} (\Real{})$,
which is a bounded open interval.

The Kullback-Leibler divergence rate from~\eqref{eqn:KLr},
when instantiated for the exponential family of Markov chains, can be expressed as,
\[
\KLr{\theta}{\lambda} =
\Lambda (\lambda) - \Lambda (\theta) - \dot{\Lambda} (\theta) (\lambda - \theta),
\]
which is convex and differentiable over $\Real{} \times \Real{}$.
Since $\dot{\Lambda} : \Real{} \to \calM$ forms a bijection from the natural parameter space, $\Real{}$, to the mean parameter space, $\calM$,
with some abuse of notation we will write
$\KLr{\mu}{\nu}$ for $\KLr{\dot{\Lambda}\inv (\mu)}{\dot{\Lambda}\inv (\nu)}$,
where $\mu, \nu \in \calM$.
Furthermore, $\KLr{\cdot}{\cdot} : \calM \times \calM \to \Real{}_{\ge 0}$ can be extended continuously, to a function
$\KLr{\cdot}{\cdot} : \Bar{\calM} \times \bar{\calM} \to \Real{}_{\ge 0} \cup \{\infty\}$, where $\Bar{\calM}$ denotes the closure of $\calM$.
This can even further be extended to a convex function on $\Real{} \times \Real{}$, by setting $\KLr{\mu}{\nu} = \infty$ if $\mu \not\in \bar{\calM}$ or $\nu \not\in \Bar{\calM}$.
For fixed $\nu \in \Real{}$, the function $\mu \mapsto \KLr{\mu}{\nu}$
is decreasing for $\mu \le \nu$ and increasing for $\mu \ge \nu$.
Similarly, for fixed $\mu \in \Real{}$, the function $\nu \mapsto \KLr{\mu}{\nu}$
is decreasing for $\nu \le \mu$ and increasing for $\nu \ge \mu$.
\section{A Maximal Inequality for Markov Chains}\label{sec:concentration}

Here we present an exponential martingale for Markov chains, which in turn leads to a maximal inequality.
For proofs, and a Chernoff bound for Markov chains we refer the interested reader to~\autoref{a:concentration}.

\begin{lemma}[Exponential martingale for Markov chains]\label{lem:martingale}
Let $\{X_n\}_{n \in \Int{}_{\ge 0}}$ be a Markov chain
over the finite state space $S$
with an irreducible transition matrix $P$ and initial distribution $q$.
Let $f : S \to \Real{}$ be a nonconstant real-valued function on the state space.
Fix $\theta \in \Real{}$ and define,
\begin{equation}\label{eqn:martingale}
M_n^\theta =
\frac{v_\theta (X_n)}{v_\theta (X_0)}
\exp\left\{\theta (f (X_1) + \ldots + f (X_n)) - n \Lambda (\theta)\right\}.
\end{equation}
Then $\{M_n^\theta\}_{n \in \Intpp{}}$ is a martingale with respect to the filtration $\{\calF_n\}_{n \in \Intpp{}}$, where $\calF_n$
is the $\sigma$-field generated by $X_0, \ldots, X_n$.
\end{lemma}
The following definition is the technical condition that we will require for our maximal inequality.

\begin{definition}[Doeblin's type of condition]\label{def:Doeblin}
Let $P$ be a transition probability matrix on the finite state space $S$.
For a nonempty set of states $A \subset S$, we say that $P$ is $A$-Doeblin
if, the submatrix of $P$ with rows and columns in $A$ is irreducible,
and for every $x \in S - A$ there exists $y \in A$ such that $P (x, y) > 0$.
\end{definition}

\addtocounter{example}{-1}
\begin{example}[continued]
For this example $P$ being $\{0\}$-Doeblin means that 
$p, q \in [0, 1)$, but already irreducibility imposed the constraints $p \in (0, 1]$ and $q \in [0, 1)$, hence the only additional constraint is $p \neq 1$.
\end{example}

\begin{remark}
Our~\autoref{def:Doeblin} is inspired by the classic Doeblin's Theorem,
see Theorem 2.2.1 in~\cite{Stroock14}.
Doeblin's Theorem states that, if the transition probability matrix $P$ 
satisfies Doeblin's condition (namely there exists $\epsilon > 0$,
and a state $y \in S$ such that for all $x \in S$ we have $P (x, y) \ge \epsilon$),
then $P$ has a unique stationary distribution $\pi$, and for all initial distributions $q$ we have geometric convergence to stationarity,
i.e. $\|q P^n - \pi\|_1 \le 2 (1-\epsilon)^n$.
Doeblin's condition, according to our~\autoref{def:Doeblin}, corresponds to $P$ being $\{y\}$-Doeblin for some $y \in S$.
\end{remark}

\begin{lemma}[Maximal inequality for irreducible Markov chains satisfying Doeblin's condition]\label{lem:maximal}
Let $\{X_n\}_{n \in \Int{}_{\ge 0}}$ be an irreducible Markov chain
over the finite state space $S$
with transition matrix $P$, initial distribution $q$,
and stationary distribution $\pi$.
Let $f : S \to \Real{}$ be a non-constant function on the state space.
Denote by $\mu (0) = \sum_{x \in S} f (x) \pi (x)$ the stationary mean when $f$ is applied, and by
$\bar{Y}_n = \frac{1}{n} \sum_{k=1}^n Y_k$ the empirical mean,
where $Y_k = f (X_k)$.
Assume that $P$ is $\left(\argmin_{x \in S} f (x)\right)$-Doeblin.
Then for all $\epsilon > 1$ we have
\[
\Pr \left(
\bigcup_{k=1}^n \left\{
\mu (0) \ge \bar{Y}_k
~\text{and}~
k \KLr{\bar{Y}_k}{\mu (0)} \ge \epsilon
\right\}
\right)
\le
C_- e \lceil \epsilon \log n \rceil e^{- \epsilon},
\]
where $C_- = C_- (P, f)$ is a positive constant depending only on the transition probability matrix $P$ and the function $f$.
\end{lemma}
\begin{remark}
If we only consider values of $\epsilon$ from a bounded subset of $(1, \infty)$, then we don't need to assume that $P$ is
$\left(\argmin_{x \in S} f (x)\right)$-Doeblin, and the constant $C_-$ will further depend on this bounded subset.
But in the analysis of the KL-UCB adaptive allocation rule we will need to consider values of $\epsilon$ that increase with the time horizon $T$,
therefore we have to impose the assumption that $P$ is $\left(\argmin_{x \in S} f (x)\right)$-Doeblin, so that $C_-$ has no dependencies on $\epsilon$.
\end{remark}
i.i.d. versions of this maximal inequality have found applicability not only in multi-armed bandit problems, but also in the case of context tree estimation,~\cite{Garivier-Leonardi-11}, indicating that our~\autoref{lem:maximal} may be of interest for other applications as well.

\section{The Round-Robin KL-UCB Adaptive Allocation Rule for Multiple Plays and Markovian Rewards}\label{sec:analysis}

For each arm $a \in [K]$ we define the empirical mean at the global time $t$ as,
\begin{equation}\label{eqn:mean-def}
\Bar{Y}_a (t) = (Y_1^a + \ldots + Y_{N_a (t)}^a)/N_a (t),
\end{equation}
and its local time counterpart as,
\[
\Bar{Y}_n^a = (Y_1^a + \ldots + Y_n^a)/n,
\]
with their link being $\Bar{Y}_n^a = \Bar{Y}_a (\tau_n^a)$,
where $\tau_n^a = \inf \{t \ge 1 : N_a (t) = n\}$.
At each round $t$ we calculate a single upper confidence bound index,
\begin{equation}\label{eqn:UCB-def}
U_a (t) =
\sup \left\{\mu \in \calM : \KL{\Bar{Y}_a (t)}{\mu} \le \frac{g (t)}{N_a (t)}\right\},
\end{equation}
where $g (t)$ is an increasing function,
and we denote its local time version by,
\[
U_n^a (t) = \sup\left\{\mu \in \calM : \KL{\Bar{Y}_n^a}{\mu} \le \frac{g (t)}{n}\right\}.
\]
Note that $U_a (t)$ is efficiently computable via a bisection method due to the monotonicity of $\KL{\Bar{Y}_a (t)}{\cdot}$.
It is straightforward to check, using the definition of $U_n^a (t)$,
the following two relations,
\begin{align}
& \Bar{Y}_n^a \le U_n^a (t) ~ \text{for all}~ n \le t \label{eqn:UCB}, \\
& U_n^a (t) ~ \text{is increasing in}~ t \ge n ~\text{for fixed}~ n. \label{eqn:incr}
\end{align}
Furthermore, in~\autoref{a:UCBs} we study the concentration properties of those upper confidence indices and of the sample means, using the concentration results for Markov chains from~\autoref{sec:concentration}.
The idea of calculating indices in a round robin way, dates back to the seminal work of~\cite{Lai-Robbins-85}.
Here we exploit this idea, which seems to have been forgotten over time in favor of algorithms that calculate indices for all the arms at each round,
and we augment it with the usage of the upper confidence bounds in~\eqref{eqn:UCB-def}, which are efficiently computable, see~\autoref{sec:simulations} for simulation results, as opposed to the statistics in Theorem 4.1 from~\cite{Ananth-Varaiya-Walrand-II-87}.
Moreover, this combination of a round-robin scheme and the indices in~\eqref{eqn:UCB-def} is amenable to a finite-time analysis, see~\autoref{a:analysis}.

\begin{algorithm}[H]
\SetAlgoLined
\caption{The round-robin KL-UCB adaptive allocation rule.}\label{alg:KL-UCB}
\kwParams{
number of arms $K \ge 2$,
time horizon $T \ge K$,
number of plays $1 \le M \le K$, \\
KL divergence rate function $\KLr{\cdot}{\cdot} : \bar{\calM} \times \bar{\calM} \to \Real{}_{\ge 0}$,
increasing function $g : \Intpp{} \to \Real{}$,
parameter $\delta \in (0, 1/K)$
}\;
\kwInit{
In the first $K$ rounds pull each arm $M$ times and set
$\Bar{Y}_a (K) = (Y_1^a + \ldots + Y_M^a)/M$, for $a = 1, \ldots, K$
}\;

\For{$t = K, \ldots, T-1$}{
Let $W_t = \{a \in [K] : N_a (t) \ge \lceil \delta t \rceil\}$\;
Pick any subset of arms $L_t \subseteq W_t$ such that:
\begin{itemize}
    \item $|L_t| = M$\;
    \item and $\dss\min_{a \in L_t} \Bar{Y}_a (t) \ge
    \sup_{b \in W_t-L_t} \Bar{Y}_b (t)$\;
\end{itemize}
Let $b \equiv t + 1 \pmod{K}$, with $b \in [K]$\;
Let
$\dss
U_b (t) =
\sup \left\{\mu \in \calM : \KLr{\Bar{Y}_b (t)}{\mu} \le
\frac{g (t)}{N_b (t)}\right\}
$\;
\eIf{$b \in L_t$ \textbf{or} $\dss\min_{a \in L_t} \Bar{Y}_a (t) \ge U_b (t)$}
{Pull the $M$ arms in $\phi_{t+1} = L_t$\;}
{Pick any $\dss a \in \argmin_{a \in L_t} \Bar{Y}_a (t)$\;
Pull the $M$ arms in $\phi_{t+1} = (L_t \setminus \{a\}) \cup \{b\}$\;
}
}
\end{algorithm}

\begin{proposition}\label{prop:well-defined}
For each $t \ge K$ we have that $|W_t| \ge M$,
and so~\autoref{alg:KL-UCB} is well defined.
\end{proposition}

\begin{theorem}[Markovian rewards and multiple plays: finite-time guarantees]\label{thm:main}
Let $P$ be an irreducible transition probability matrix on the finite state space $S$,
and $f : S \to \Real{}$ be a real-valued reward function, such that
$P$ is $\left(\argmin_{x \in S} f (x)\right)$-Doeblin.
Assume that the $K$ arms correspond to the parameter configuration $\btheta \in \Real{K}$ of the exponential family of Markov chains,
as described in~\autoref{eqn:exp-fam-MC}.
Without loss of generality assume that the $K$ arms are ordered so that,
\[
\mu (\theta_1) \ge \ldots \ge \mu (\theta_N) > 
\mu (\theta_{N+1}) \ldots = \mu (\theta_M) = \ldots = \mu (\theta_L) >
\mu (\theta_{L+1}) \ge \ldots \ge \mu (\theta_K).
\]
Fix $\epsilon \in \left(0, \min (\mu (\theta_N) - \mu (\theta_M), \mu (\theta_M) - \mu (\theta_{L+1}))\right)$.
The KL-UCB adaptive allocation rule for Markovian rewards and multiple plays,~\autoref{alg:KL-UCB}, 
with the choice $g (t) = \log t + 3 \log \log t$, enjoys the following finite-time upper bound on the regret,
\[
R_\btheta^\bphi (T) \le
\sum_{b=L+1}^K
\frac{\mu (\theta_M) -\mu (\theta_b)}
{\KL{\mu (\theta_b)}{\mu (\theta_M) - \epsilon}} \log T 
+ c_1 \sqrt{\log T} 
+ c_2 \log \log T + c_3 \sqrt{\log \log T} + c_4,
\]
where $c_1, c_2, c_3, c_4$ are constants with respect to $T$,
which are given more explicitly in the analysis.
\end{theorem}

\begin{corollary}[Asymptotic optimality]\label{cor:optimal}
In the context of~\autoref{thm:main} the KL-UCB adaptive allocation rule,~\autoref{alg:KL-UCB}, is asymptotically optimal, and,
\[
\lim_{T \to \infty} \frac{R_\btheta^\bphi (T)}{\log T} =
\sum_{b=L+1}^K
\frac{\mu (\theta_M) -\mu (\theta_b)}
{\KL{\mu (\theta_b)}{\mu (\theta_M)}}.
\]
\end{corollary}
\section{The Round-Robin KL-UCB Adaptive Allocation Rule for Multiple Plays and i.i.d. Rewards}

As a byproduct of our work in~\autoref{sec:analysis} we further obtain a finite-time regret bound, which is asymptotically optimal, for the case of
multiple plays and i.i.d. rewards, from an exponential family of probability densities.

We first review the notion of an exponential family of probability densities, for which the standard reference is~\cite{Brown86}.
Let $(X, \calX, \rho)$ be a probability space.
A one-parameter exponential family is a family of probability densities
$\{p_\theta : \theta \in \Theta\}$ with respect to the measure $\rho$ on $X$, of the form,
\begin{equation}\label{eqn:exp-fam}
p_\theta (x) = \exp\{\theta f (x) - \Lambda (\theta)\} h (x),
\end{equation}
where $f : X \to \Real{}$ is called the sufficient statistic, is $\calX$-measurable, and there is no $c \in \Real{}$ such that $f (x) \stackrel{\rho - a.s.}{=} c, ~ h : X \to \Realp{}$ is called the carrier density,
and is a density with respect to $\rho$, and $\Lambda$ is called the log-Moment-Generating-Function and is given by 
$\Lambda (\theta) = \log \int_X e^{\theta f (x)} h (x) \rho (dx)$, which is finite for $\theta$ in the natural parameter space $\Theta = \{\theta \in \Real{} : \int_X e^{\theta f (x)} h (x) \rho (dx) < \infty\}$.
The log-MGF, $\Lambda (\theta)$, is strictly convex and its derivative forms a bijection between the natural parameters, $\theta$,
and the mean parameters, $\mu (\theta) = \int_X f (x) p_\theta (x) \rho (d x)$. The Kullback-Leibler divergence between $p_\theta$ and $p_\lambda$, for $\theta, \lambda \in \Theta$, can be written as $\KL{\theta}{\lambda} = \Lambda (\lambda) - \Lambda (\theta) - \dot{\Lambda} (\theta) (\lambda - \theta)$.

For this section, each arm $a \in [K]$ with parameter $\theta_a$ corresponds to the i.i.d. process $\{X_n^a\}_{n \in \Intpp{}}$, where each $X_n^a$ has density $p_{\theta_a}$ with respect to $\rho$, which gives rise to the i.i.d. reward process $\{Y_n^a\}_{n \in \Intpp{}}$, with $Y_n^a = f (X_n^a)$.

\begin{remark}\label{rmk:IID-Markov}
When there is a finite set $S \in \calX$ such that $\rho (S) = 1$,
then the exponential family of probability densities in~\autoref{eqn:exp-fam}, is just a special case of the exponential family of Markov chains in~\autoref{eqn:exp-fam-MC}, as can be seen by setting $P (x, \cdot) = h (\cdot)$, for all $x \in S$. Then $v_\theta (x) = 1$ for all $x \in S$, the log-Perron-Frobenius eigenvalue coincides with the log-MGF, and $\Theta = \Real{}$.
Therefore,~\autoref{thm:main} already resolves the case of multiple plays and i.i.d. rewards from an exponential family of finitely supported densities.
\end{remark}

\begin{theorem}[i.i.d. rewards and multiple plays: finite-time guarantees]\label{thm:main-IID}
Let $(X, \calX, \rho)$ be a probability space, 
$f : X \to \Real{}$ a $\calX$-measurable function, and $h : X \to \Realp{}$
a density with respect to $\rho$.
Assume that the $K$ arms correspond to the parameter configuration $\btheta \in \Theta^K$ of the exponential family of probability densities, as described in~\autoref{eqn:exp-fam}.
Without loss of generality assume that the $K$ arms are ordered so that,
\[
\mu (\theta_1) \ge \ldots \ge \mu (\theta_N) > 
\mu (\theta_{N+1}) \ldots = \mu (\theta_M) = \ldots = \mu (\theta_L) >
\mu (\theta_{L+1}) \ge \ldots \ge \mu (\theta_K).
\]
Fix $\epsilon \in \left(0, \min (\mu (\theta_N) - \mu (\theta_M), \mu (\theta_M) - \mu (\theta_{L+1}))\right)$.
The KL-UCB adaptive allocation rule for i.i.d. rewards and multiple plays,~\autoref{alg:KL-UCB}, 
with the choice $g (t) = \log t + 3 \log \log t$, enjoys the following finite-time upper bound on the regret,
\[
R_\btheta^\bphi (T) \le
\sum_{b=L+1}^K
\frac{\mu (\theta_M) -\mu (\theta_b)}
{\KL{\mu (\theta_b)}{\mu (\theta_M) - \epsilon}} \log T 
+ c_1 \sqrt{\log T} 
+ c_2 \log \log T + c_3 \sqrt{\log \log T} + c_4,
\]
where $c_1, c_2, c_3, c_4$ are constants with respect to $T$.

Consequently, the KL-UCB adaptive allocation rule,~\autoref{alg:KL-UCB}, is asymptotically optimal, and,
\[
\lim_{T \to \infty} \frac{R_\btheta^\bphi (T)}{\log T} =
\sum_{b=L+1}^K
\frac{\mu (\theta_M) -\mu (\theta_b)}
{\KL{\mu (\theta_b)}{\mu (\theta_M)}}.
\]
\end{theorem}

\begin{remark}\label{rmk:single-plays}
For the special case of single plays, $M=1$,
such a finite-time regret bound is derived in~\cite{Cappe-Garivier-Maillard-Munos-Stoltz-13}, and here we generalize it for multiple plays, $1 \le M  \le K$. One striking difference is that we consider calculations of KL upper confidence bounds in a round-robin way, as opposed to calculating them for all the arms at each round. 
But computing KL-UCB indices adds an extra computational overhead, as it entails inverting an increasing function via the bisection method.
Thus, our approach has important practical implications as it leads to significantly more efficient algorithms.
We verify this via simulations in~\autoref{sec:simulations}.
\end{remark}
\section{Simulation Results}\label{sec:simulations}

In the context of~\autoref{e:example},
we set $p = 0.49, q = 0.45, K = 14$, and $T = 10^6$.
We generated the bandit instance $\theta_1, \ldots, \theta_K$ by drawing i.i.d. $N (0, 1/16)$ samples.
Four adaptive allocation rules were taken into consideration:
\begin{enumerate}
    \item \textbf{UCB}: at reach round calculate all UCB indices,
    \[
    U_a^{\mathrm{UCB}} (t) = \Bar{Y}_a (t) + \beta \sqrt{\frac{2 \log t}{N_a (t)}},
    ~\text{for}~ a = 1, \ldots, K.
    \]
    
    \item \textbf{Round-Robin UCB}: at reach round calculate a single UCB index,
    \[
    U_b^{\mathrm{UCB}} (t) = \Bar{Y}_b (t) + \beta \sqrt{\frac{2 \log t}{N_b (t)}},
    ~\text{only for}~ b \equiv t + 1 \pmod{K}.
    \]
    
    \item \textbf{KL-UCB}: at reach round calculate all KL-UCB indices,
    \[
    U_a^{\mathrm{KL-UCB}} (t) =
    \sup \left\{\mu \in \calM : \KL{\Bar{Y}_a (t)}{\mu} \le \frac{\log t + 3 \log \log t}{N_a (t)}\right\},
    ~\text{for}~ a = 1, \ldots, K.
    \]
    
    \item \textbf{Round-Robin KL-UCB}: at reach round calculate a single KL-UCB index,
    \[
    U_b^{\mathrm{KL-UCB}} (t) =
    \sup \left\{\mu \in \calM : \KL{\Bar{Y}_b (t)}{\mu} \le \frac{\log t + 3 \log \log t}{N_b (t)}\right\},
    ~\text{only for}~ b \equiv t + 1 \pmod{K}.
    \]
\end{enumerate}
For the UCB indices, after some tuning, we picked $\beta = 1$ which is significantly smaller than the theoretical values of $\beta$ from~\cite{Tekin-Liu-10,Tekin-Liu-12,Moulos-Hoeffding-20}.
For each of those adaptive allocation rules $10^4$ Monte Carlo iterations were performed in order to estimate
the expected regret, and the simulation results are presented in the following plots.

\begin{figure}[H]
  \centering
  \begin{minipage}[b]{0.45\textwidth}
    \includegraphics[width=\textwidth]{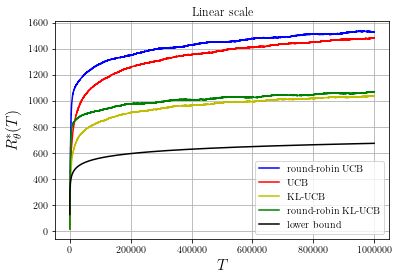}
    \caption{Regret of the various algorithms as a function of time in linear scale.}
  \end{minipage}
  \hfill
  \begin{minipage}[b]{0.45\textwidth}
    \includegraphics[width=\textwidth]{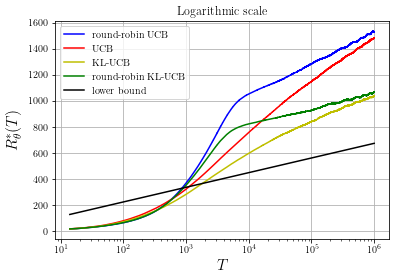}
    \caption{Regret of the various algorithms as a function of time in logarithmic scale.}
  \end{minipage}
\end{figure}

For our simulations we used the programming language C, to produce highly efficient code,
and a personal computer with a 2.6GHz processor and 16GB of memory.
We report that the simulation for the Round-Robin KL-UCB adaptive allocation rule was $14.48$ times faster than the
simulation for the KL-UCB adaptive allocation rule. 
This behavior is expected since each calculation of a KL-UCB index induces a significant computation cost as it involves finding the inverse of an increasing function using the bisection method.
Additionally, the simulation for the Round-Robin UCB adaptive allocation rule was $3.15$ times faster than the simulation for the KL-UCB adaptive allocation rule, and this is justified from the fact that calculating mathematical functions such as $\log(\cdot)$ and $\sqrt{\cdot}$, is more costly than calculating averages which only involve a division. Our simulation results yield that in practice round-robin schemes are significantly faster than schemes that calculate the indices of all the arms at each round, and the computational gap is increasing with the number of arms $K$, while the behavior of the expected regrets is very similar.

\section*{Acknowledgements} We would like to thank 
Venkat Anantharam, Jim Pitman and 
Satish Rao for many helpful discussions.
This research was supported in part by the NSF grant CCF-1816861.

\bibliographystyle{apalike}
\bibliography{references}

\begin{thebibliography}{}

\bibitem[Agrawal, 1995]{Agrawal-95}
Agrawal, R. (1995).
\newblock Sample mean based index policies with {$O(\log n)$} regret for the
  multi-armed bandit problem.
\newblock {\em Adv. in Appl. Probab.}, 27(4):1054--1078.

\bibitem[Anantharam et~al., 1987a]{Ananth-Varaiya-Walrand-I-87}
Anantharam, V., Varaiya, P., and Walrand, J. (1987a).
\newblock Asymptotically efficient allocation rules for the multiarmed bandit
  problem with multiple plays. {I}. {I}.{I}.{D}. rewards.
\newblock {\em IEEE Trans. Automat. Control}, 32(11):968--976.

\bibitem[Anantharam et~al., 1987b]{Ananth-Varaiya-Walrand-II-87}
Anantharam, V., Varaiya, P., and Walrand, J. (1987b).
\newblock Asymptotically efficient allocation rules for the multiarmed bandit
  problem with multiple plays. {II}. {M}arkovian rewards.
\newblock {\em IEEE Trans. Automat. Control}, 32(11):977--982.

\bibitem[Auer et~al., 2002]{Auer-Cesa-Fischer-02}
Auer, P., Cesa-Bianchi, N., and Fischer, P. (2002).
\newblock Finite-time {A}nalysis of the {M}ultiarmed {B}andit {P}roblem.
\newblock {\em Mach. Learn.}, 47(2-3):235--256.

\bibitem[Brown, 1986]{Brown86}
Brown, L.~D. (1986).
\newblock {\em Fundamentals of statistical exponential families with
  applications in statistical decision theory}, volume~9 of {\em Institute of
  Mathematical Statistics Lecture Notes---Monograph Series}.
\newblock Institute of Mathematical Statistics, Hayward, CA.

\bibitem[Bubeck and Cesa-Bianchi, 2012]{BC12}
Bubeck, S. and Cesa-Bianchi, N. (2012).
\newblock {Regret Analysis of Stochastic and Nonstochastic Multi-armed Bandit
  Problems}.
\newblock {\em Foundations and Trends® in Machine Learning}, 5(1):1--122.

\bibitem[Capp\'{e} et~al., 2013]{Cappe-Garivier-Maillard-Munos-Stoltz-13}
Capp\'{e}, O., Garivier, A., Maillard, O.-A., Munos, R., and Stoltz, G. (2013).
\newblock Kullback-{L}eibler upper confidence bounds for optimal sequential
  allocation.
\newblock {\em Ann. Statist.}, 41(3):1516--1541.

\bibitem[Chazottes et~al., 2007]{Chazottes-Collet-Kulske-Redig-07}
Chazottes, J.-R., Collet, P., K\"{u}lske, C., and Redig, F. (2007).
\newblock Concentration inequalities for random fields via coupling.
\newblock {\em Probab. Theory Related Fields}, 137(1-2):201--225.

\bibitem[Combes and Proutiere, 2014]{Combes-Proutiere-14}
Combes, R. and Proutiere, A. (2014).
\newblock Unimodal bandits without smoothness.

\bibitem[Cover and Thomas, 2006]{CovThom06}
Cover, T.~M. and Thomas, J.~A. (2006).
\newblock {\em Elements of information theory}.
\newblock Wiley-Interscience [John Wiley \& Sons], Hoboken, NJ, second edition.

\bibitem[Durrett, 2019]{Durrett19}
Durrett, R. (2019).
\newblock {\em {Probability: Theory and Examples}}.
\newblock Cambridge University Press, Cambridge, fifth edition.

\bibitem[Garivier and Capp\'{e}, 2011]{Garivier-Cappe-11}
Garivier, A. and Capp\'{e}, O. (2011).
\newblock The {KL-UCB} {A}lgorithm for {B}ounded {S}tochastic {B}andits and
  {B}eyond.
\newblock In Kakade, S.~M. and von Luxburg, U., editors, {\em Proceedings of
  the 24th Annual Conference on Learning Theory}, volume~19 of {\em Proceedings
  of Machine Learning Research}, pages 359--376, Budapest, Hungary. PMLR.

\bibitem[Garivier and Leonardi, 2011]{Garivier-Leonardi-11}
Garivier, A. and Leonardi, F. (2011).
\newblock Context tree selection: A unifying view.
\newblock {\em Stochastic Processes and their Applications}, 121(11):2488 --
  2506.

\bibitem[Gillman, 1993]{Gillman93}
Gillman, D. (1993).
\newblock A {C}hernoff bound for random walks on expander graphs.
\newblock In {\em 34th {A}nnual {S}ymposium on {F}oundations of {C}omputer
  {S}cience ({P}alo {A}lto, {CA}, 1993)}, pages 680--691. IEEE Comput. Soc.
  Press, Los Alamitos, CA.

\bibitem[Glynn and Ormoneit, 2002]{Glynn-Ormoneir-02}
Glynn, P.~W. and Ormoneit, D. (2002).
\newblock Hoeffding's inequality for uniformly ergodic {M}arkov chains.
\newblock {\em Statist. Probab. Lett.}, 56(2):143--146.

\bibitem[Horn and Johnson, 2013]{HJ13}
Horn, R.~A. and Johnson, C.~R. (2013).
\newblock {\em Matrix analysis}.
\newblock Cambridge University Press, Cambridge, second edition.

\bibitem[Kaufmann et~al., 2016]{Kaufmann-Cappe-Garivier-16}
Kaufmann, E., Capp{\'e}, O., and Garivier, A. (2016).
\newblock {On the Complexity of Best-arm Identification in Multi-armed Bandit
  Models}.
\newblock {\em J. Mach. Learn. Res.}, 17(1):1--42.

\bibitem[Komiyama et~al., 2015]{KHN-15}
Komiyama, J., Honda, J., and Nakagawa, H. (2015).
\newblock {Optimal Regret Analysis of Thompson Sampling in Stochastic
  Multi-Armed Bandit Problem with Multiple Plays}.
\newblock In {\em Proceedings of the 32nd International Conference on
  International Conference on Machine Learning - Volume 37}, ICML’15, page
  1152–1161. JMLR.org.

\bibitem[Kontorovich and Ramanan, 2008]{Kontorovich-Ramanan-08}
Kontorovich, L. and Ramanan, K. (2008).
\newblock Concentration inequalities for dependent random variables via the
  martingale method.
\newblock {\em Ann. Probab.}, 36(6):2126--2158.

\bibitem[Lai and Robbins, 1985]{Lai-Robbins-85}
Lai, T.~L. and Robbins, H. (1985).
\newblock Asymptotically efficient adaptive allocation rules.
\newblock {\em Adv. in Appl. Math.}, 6(1):4--22.

\bibitem[Lattimore and Szepesv\'{a}ri, 2019]{Lat-Szep-19}
Lattimore, T. and Szepesv\'{a}ri, C. (2019).
\newblock {Bandit Algorithms}.

\bibitem[Lezaud, 1998]{Lezaud98}
Lezaud, P. (1998).
\newblock Chernoff-type bound for finite {M}arkov chains.
\newblock {\em Ann. Appl. Probab.}, 8(3):849--867.

\bibitem[Maillard et~al., 2011]{Maillard-Munos-Stoltz-11}
Maillard, O.-A., Munos, R., and Stoltz, G. (2011).
\newblock A {F}inite-{T}ime {A}nalysis of {M}ulti-armed {B}andits {P}roblems
  with {K}ullback-{L}eibler divergences.
\newblock In Kakade, S.~M. and von Luxburg, U., editors, {\em Proceedings of
  the 24th Annual Conference on Learning Theory}, volume~19 of {\em Proceedings
  of Machine Learning Research}, pages 497--514, Budapest, Hungary. PMLR.

\bibitem[Marton, 1996a]{Marton-96-a}
Marton, K. (1996a).
\newblock Bounding {$\overline d$}-distance by informational divergence: a
  method to prove measure concentration.
\newblock {\em Ann. Probab.}, 24(2):857--866.

\bibitem[Marton, 1996b]{Marton-96-b}
Marton, K. (1996b).
\newblock A measure concentration inequality for contracting {M}arkov chains.
\newblock {\em Geom. Funct. Anal.}, 6(3):556--571.

\bibitem[Marton, 1998]{Marton-98}
Marton, K. (1998).
\newblock Measure concentration for a class of random processes.
\newblock {\em Probab. Theory Related Fields}, 110(3):427--439.

\bibitem[Marton, 2003]{Marton-03}
Marton, K. (2003).
\newblock Measure concentration and strong mixing.
\newblock {\em Studia Sci. Math. Hungar.}, 40(1-2):95--113.

\bibitem[Miller, 1961]{Miller61}
Miller, H.~D. (1961).
\newblock A convexity property in the theory of random variables defined on a
  finite {M}arkov chain.
\newblock {\em Ann. Math. Statist.}, 32:1260--1270.

\bibitem[Moulos, 2019]{Moulos19-bandits-identification}
Moulos, V. (2019).
\newblock {Optimal Best Markovian Arm Identification with Fixed Confidence}.
\newblock In {\em Advances in Neural Information Processing Systems (NeurIPS)},
  pages 5605--5614.

\bibitem[Moulos, 2020]{Moulos-Hoeffding-20}
Moulos, V. (2020).
\newblock {A Hoeffding Inequality for Finite State Markov Chains and its
  Applications to Markovian Bandits}.
\newblock In {\em IEEE International Symposium on Information Theory (ISIT)}.

\bibitem[Moulos and Anantharam, 2019]{Moulos-Ananth-19}
Moulos, V. and Anantharam, V. (2019).
\newblock {Optimal Chernoff and Hoeffding Bounds for Finite State Markov
  Chains}.

\bibitem[Ortner et~al., 2012]{ORAM-12}
Ortner, R., Ryabko, D., Auer, P., and R\'{e}mi, M. (2012).
\newblock {Regret Bounds for Restless Markov Bandits.}
\newblock In {\em {Algorithmic Learning Theory (ALT)}}.

\bibitem[Paulin, 2015]{Paulin15}
Paulin, D. (2015).
\newblock Concentration inequalities for {M}arkov chains by {M}arton couplings
  and spectral methods.
\newblock {\em Electron. J. Probab.}, 20:no. 79, 32.

\bibitem[Samson, 2000]{Samson-00}
Samson, P.-M. (2000).
\newblock Concentration of measure inequalities for {M}arkov chains and
  {$\Phi$}-mixing processes.
\newblock {\em Ann. Probab.}, 28(1):416--461.

\bibitem[Slivkins, 2019]{Slivkins-19}
Slivkins, A. (2019).
\newblock {Introduction to Multi-Armed Bandits}.
\newblock {\em Foundations and Trends® in Machine Learning}, 12(1-2):1--286.

\bibitem[Stroock, 2014]{Stroock14}
Stroock, D.~W. (2014).
\newblock {\em An introduction to {M}arkov processes}, volume 230 of {\em
  Graduate Texts in Mathematics}.
\newblock Springer, Heidelberg, second edition.

\bibitem[{Tekin} and {Liu}, 2010]{Tekin-Liu-10}
{Tekin}, C. and {Liu}, M. (2010).
\newblock {Online algorithms for the multi-armed bandit problem with Markovian
  rewards}.
\newblock In {\em 2010 48th Annual Allerton Conference on Communication,
  Control, and Computing (Allerton)}, pages 1675--1682.

\bibitem[Tekin and Liu, 2012]{Tekin-Liu-12}
Tekin, C. and Liu, M. (2012).
\newblock Online learning of rested and restless bandits.
\newblock {\em IEEE Trans. Inf. Theor.}, 58(8):5588--5611.

\bibitem[Ville, 1939]{Ville39}
Ville, J. (1939).
\newblock {\em \'{E}tude critique de la notion de collectif}.
\newblock NUMDAM.

\end{thebibliography}

\appendix

\begin{appendices}
\section{Concentration Lemmata for Markov Chains}\label{a:concentration}

We first develop a Chernoff bound, which remarkably does not impose any conditions on the Markov chain other than irreducibility, which is though a mandatory requirement for the stationary mean to be well-defined.

\begin{lemma}[Chernoff bound for irreducible Markov chains]\label{lem:Chernoff}
Let $\{X_n\}_{n \in \Int{}_{\ge 0}}$ be an irreducible Markov chain
over the finite state space $S$
with transition probability matrix $P$, initial distribution $q$,
and stationary distribution $\pi$.
Let $f : S \to \Real{}$ be a nonconstant function on the state space.
Denote by $\mu (0) = \sum_{x \in S} f (x) \pi (x)$ the stationary mean when $f$ is applied, and by
$\bar{Y}_n = \frac{1}{n} \sum_{k=1}^n Y_k$ the empirical mean,
where $Y_k = f (X_k)$.
Let $F$ be a closed subset of $\calM \cap [\mu (0), \infty)$. Then,
\[
\Pr \left( \bar{Y}_n \ge \mu \right)
\le
C_+ e^{- n \KLr{\mu}{\mu (0)}},~\text{for}~ \mu \in F,
\]
where $\KLr{\cdot}{\cdot}$ stands for the Kullback-Leibler divergence rate
in the exponential family of stochastic matrices generated by $P$ and $f$,
and $C_+ = C_+ (P, f, F)$ is a positive constant depending only on the transition probability matrix $P$, the function $f$ and the closed set $F$.
\end{lemma}
\begin{proof}[Proof of~\autoref{lem:Chernoff}.] \hfill\break
Using the standard exponential transform followed by Markov's inequality we obtain that for any $\theta \ge 0$,
\[
\Pr (\Bar{Y}_n \ge \mu)
\le \Pr (e^{n \theta \Bar{Y}_n} \ge e^{n \theta \mu})
\le \exp\left\{-n \left(\theta \mu - 
\frac{1}{n} \log \E \left[e^{\theta (f (X_1) + \ldots + f (X_n))}\right]\right)\right\}.
\]
We can upper bound the expectation from above in the following way,
\begin{align*}
\E \left[e^{\theta (f (X_1) + \ldots + f (X_n))}\right] 
&=
\sum_{x_0, \ldots, x_n \in S} q (x_0) P (x_0, x_1) e^{\theta f (x_1)}
\ldots 
P (x_{n-1}, x_n) e^{\theta f (x_n)} \\
&= \sum_{x_0, x_n \in S} q (x_0) \Tilde{P}_\theta^n (x_0, x_n) \\
&\le \frac{1}{\min_{x \in S} v_\theta (x)} \sum_{x_0, x_n \in S} q (x_0) \Tilde{P}_\theta^n (x_0, x_n) v_\theta (x_n) \\
&= \frac{\rho (\theta)^n}{\min_{x \in S} v_\theta (x)} \sum_{x_0 \in S} q (x_0) v_\theta (x_0) \\
&\le \max_{x, y \in S} \frac{v_\theta (y)}{v_\theta (x)} \rho (\theta)^n,
\end{align*}
where in the last equality we used the fact that $v_\theta$ is a right Perron-Frobenius eigenvector of $\Tilde{P}_\theta$.

From those two we obtain,
\[
\Pr (\Bar{Y}_n \ge \mu) \le 
\max_{x, y \in S} \frac{v_\theta (y)}{v_\theta (x)} 
\exp\left\{-n (\theta \mu - \Lambda (\theta))\right\},
\]
and if we plug in $\theta_\mu = \Dot{\Lambda}\inv (\mu)$,
which is a nonnegative real number since $\mu \in F \subseteq \calM \cap [\mu (0), \infty)$, we obtain,
\[
\Pr (\Bar{Y}_n \ge \mu) \le 
\max_{x, y \in S} \frac{v_{\theta_\mu} (y)}{v_{\theta_\mu} (x)} 
\exp\left\{-n \KLr{\mu}{\mu (0)}\right\},
\]
We assumed that $F$ is closed, and moreover $F$ is bounded since it is a subset of the bounded open interval $\calM$.
Therefore, $F$ is compact, and so $\dot{\Lambda}\inv (F)$ is compact as well. 
Then due to the fact that $\theta \mapsto v_\theta (x)/v_\theta (y)$ is continuous, from Lemma 2 in~\cite{Moulos-Ananth-19},
we deduce that,
\[
\sup_{\theta \in \dot{\Lambda}\inv (F)} \max_{x, y \in S} \frac{v_\theta (y)}{v_\theta (x)} < \infty,
\]
which we define to be the finite constant $C_+$ of~\autoref{lem:Chernoff}, and which may only depend on $P, f$ and $F$.
\end{proof}
\begin{remark}\label{rmk:Chernoff-comparison}
This bound is a variant of Theorem 1 in~\cite{Moulos-Ananth-19}, where the authors 
derive a Chernoff bound under some structural assumptions on the transition probability matrix $P$ and the function $f$.
In our~\autoref{lem:Chernoff}, following their techniques, we derive a Chernoff bound
without any assumptions, relying though on the fact that $\mu$ lies in a closed subset of the mean parameter space.
\end{remark}

Next, we proceed with the proofs of the lemmata in~\autoref{sec:concentration}.

\begin{proof}[Proof of~\autoref{lem:martingale}.] \hfill\break
\begin{align*}
\E (M_{n+1}^\theta \mid \calF_n)
&=
M_n^\theta
\frac{e^{- \Lambda (\theta)}}{v_\theta (X_n)}
\E (v_\theta (X_{n+1}) e^{\theta f (X_{n+1})} \mid \calF_n) \\
&=
M_n^\theta
\frac{e^{- \Lambda (\theta)}}{v_\theta (X_n)}
\sum_{x \in S} v_\theta (x) e^{\theta f (x)} P (X_n, y) \\
&=
M_n^\theta
\frac{e^{- \Lambda (\theta)}}{v_\theta (X_n)}
\sum_{x \in S} \Tilde{P}_\theta (X_n, x) v_\theta (x) \\
&=
M_n^\theta,
\end{align*}
where in the last equality we used the fact that $v_\theta$ is a right Perron-Frobenius eigenvector of $\Tilde{P}_\theta$.
\end{proof}

\begin{proof}[Proof of~\autoref{lem:maximal}.] \hfill\break
Our proof extends the argument from Lemma 11 in
~\cite{Cappe-Garivier-Maillard-Munos-Stoltz-13}, which deals with IID random variables.
In order to handle the Markovian dependence we need to use the exponential martingale for Markov chains from~\autoref{lem:martingale}, as well as continuity results for the right Perron-Frobenius eigenvector.

Following the proof strategy used to establish the law of the iterated logarithm, we split the range of the union $[n]$ into chunks of exponentially increasing sizes.
Denote by $\alpha > 1$ the growth factor, to be specified later,
and let $n_m = \lfloor \alpha^m \rfloor$ be the end point of the $m$-th chunk, with $n_0 = 0$. An upper bound on the number of chunks is
$M = \lceil \log n / \log \alpha \rceil$,
and so we have that
\begin{align*}
\bigcup_{k=1}^n \left\{
\mu (0) \ge \bar{Y}_k, ~
k \KLr{\bar{Y}_k}{\mu (0)} \ge \epsilon
\right\}
& \subseteq
\bigcup_{m=1}^M
\bigcup_{k = n_{m-1}+1}^{n_m} \left\{
\mu (0) \ge \bar{Y}_k, 
k \KLr{\bar{Y}_k}{\mu (0)} \ge \epsilon
\right\} \\
& \subseteq 
\bigcup_{m=1}^M 
\bigcup_{k = n_{m-1}+1}^{n_m} \left\{
\mu (0) \ge \bar{Y}_k, 
\KLr{\bar{Y}_k}{\mu (0)} \ge \frac{\epsilon}{n_m}
\right\}.
\end{align*}
Let $\mu_m = \inf \{\mu < \mu (0) : \KL{\mu}{\mu (0)} \le \epsilon/n_m \}$, and $\theta_m = \dot{\Lambda}\inv (\mu_m) < \dot{\Lambda}\inv (\mu (0)) = 0$
so that $\theta_m \mu_m - \Lambda (\theta_m) = \KL{\mu_m}{\mu (0)}$.
Then,
\begin{align*}
\left\{
\mu (0) \ge \bar{Y}_k, ~
\KL{\bar{Y}_k}{\mu (0)} \ge \frac{\epsilon}{n_m}
\right\}
&\subseteq \left\{
\bar{Y}_k \le \mu_m
\right\} \\
&= 
\left\{
e^{\theta_m k \Bar{Y}_k - k \Lambda (\theta_m)} \ge 
e^{k (\theta_m \mu_m - \Lambda (\theta_m))}
\right\} \\
&= 
\left\{
M_k^{\theta_m} \ge 
\frac{v_{\theta_m} (X_k)}{v_{\theta_m} (X_0)}
e^{k \KL{\mu_m}{\mu (0)}}
\right\} \\
&\subseteq 
\left\{
M_k^{\theta_m} \ge 
\frac{v_{\theta_m} (X_k)}{v_{\theta_m} (X_0)}
e^{(n_{m-1}+1) \KL{\mu_m}{\mu (0)}}
\right\}.
\end{align*}
At this point we use the assumption that $P$ is $\left(\argmin_{x \in S} f (x) \right)$-Doeblin in order to invoke Proposition 1 from~\cite{Moulos-Ananth-19}, which in our setting states that there exists a constant
$C_- = C_- (P, f) \ge 1$ such that,
\[
\frac{1}{C_-} \le 
\inf_{\theta \in \Real{}_{\le 0}, x, y \in S}
\frac{v_\theta (y)}{v_\theta (x)}.
\]

This gives us the inclusion,
\[
\left\{
M_k^{\theta_m} \ge 
\frac{v_{\theta_m} (X_k)}{v_{\theta_m} (X_0)}
e^{(n_{m-1}+1) \KL{\mu_m}{\mu (0)}}
\right\}
\subseteq 
\left\{
M_k^{\theta_m} \ge 
\frac{e^{(n_{m-1}+1) \KL{\mu_m}{\mu (0)}}}{C_-}
\right\}.
\]
In~\autoref{lem:martingale} we have established that
$M_k^{\theta_m}$ is a positive martingale,
which combined with a maximal inequality for martingales due to~\cite{Ville39} (see Exercise 4.8.2 in~\cite{Durrett19} for a modern reference), yields that,
\begin{align*}
\Pr \left(
\bigcup_{k=n_{m-1}+1}^{n_m}\left\{
M_k^{\theta_m} \ge 
\frac{e^{(n_{m-1}+1) \KL{\mu_m}{\mu (0)}}}{C_-}
\right\}\right)
&\le
C_- e^{-(n_{m-1}+1) \KL{\mu_m}{\mu (0)}} \\
&\le 
C_- e^{-\epsilon\frac{n_{m-1}+1}{n_m}} \le 
C_- e^{-\frac{\epsilon}{\alpha}}.
\end{align*}
To conclude, we pick the growth factor $\alpha = \epsilon/(\epsilon-1)$,
and we upper bound the number of chunks by
$M \le \lceil \epsilon \log n \rceil$.
\end{proof}
\section{Concentration Properties of Upper Confidence Bounds and Sample Means}\label{a:UCBs}

\begin{lemma}
    For every arm $a = 1, \ldots, K$, and $t \ge 3$, we have that,
    \begin{equation}\label{eqn:maximal}
        \Pr_{\theta_a} \left(\min_{n=1,\ldots,t} U_n^a (t) \le \mu (\theta_a)\right) \le \frac{4 e C_-^a}{t \log t},
    \end{equation}
    where $C_-^a$ is the constant prescribed in~\autoref{lem:maximal},
    when the maximal inequality is applied to the Markov chain with parameter $\theta_a$.
\end{lemma}
\begin{proof}
    \begin{align*}
    \Pr_{\theta_a} \left(
    \min_{n=1,\ldots,t} U_n^a (t) \le \mu (\theta_a)\right)
    &\le \Pr_{\theta_a} \left(
    \bigcup_{n=1}^t \{\mu (\theta_a) > \Bar{Y}_n^a ~\text{and}~
    n \KL{\Bar{Y}_n^a}{\mu (\theta_a)} \ge g (t)\}\right) \\
    &\le C_-^a e \lceil g (t) \log t \rceil e^{- g (t)} 
    \le 4 C_-^a e (\log t)^2 e^{- g (t)}
    = \frac{4 e C_-^a}{t \log t},
    \end{align*}
    where for the first inequality we used~\autoref{eqn:UCB} and the definition of $U_n^a (t)$, while for the second inequality we used~\autoref{lem:maximal}.
\end{proof}

\begin{lemma}
    For every arm $a = 1, \ldots, K$, and for $\mu (\lambda) > \mu (\theta_a)$,
    \begin{align}
    \sum_{n=1}^\infty \Pr_{\theta_a} (\mu (\lambda) \le U_n^a (T))
    &\le
    \frac{g (T)}{\KL{\mu (\theta_a)}{\mu (\lambda)}} + 
    1 +
    8 \sigma_{\theta_a,\lambda}^2
    \left(
    \frac{\KLrd{\mu (\theta_a)}{\mu (\lambda)}}
    {\KLr{\mu (\theta_a)}{\mu (\lambda)}}\right)^2
    \label{eqn:sum-Chernoff} \\
    & \qquad + 2 \sqrt{2 \pi \sigma_{\theta_a,\lambda}^2} \sqrt{\frac{\KLrd{\mu (\theta_a)}{\mu (\lambda)}^2}{\KLr{\mu (\theta_a)}{\mu (\lambda)}^3}} \sqrt{g (T)},\nonumber
    \end{align}
    where $\sigma_{\theta, \lambda}^2 = \sup_{\theta \in [\theta_a, \lambda]} \ddot{\Lambda} (\theta) \in (0, \infty)$, and
    $\KLd{\mu (\theta_a)}{\mu (\lambda)} = \frac{d \KL{\mu}{\mu (\lambda)}}{d \mu} \mid_{\mu = \mu (\theta_a)}$.
\end{lemma}
\begin{proof}
    The proof is based on the argument given in Appendix A.2 of~\cite{Cappe-Garivier-Maillard-Munos-Stoltz-13}, adapted though for the case of Markov chains.
    If $\mu (\lambda) \le U_n^a (T)$, and $\Bar{Y}_n^a \le \mu (\lambda)$, then $\KL{\Bar{Y}_n^a}{\mu (\lambda)} \le g (T)/n$.
    Let $\mu_x = \inf \{\mu \le \mu (\lambda) : \KL{\mu}{\mu (\lambda)} \le x\}$. This in turn implies that
    $\KL{\Bar{Y}_n^a}{\mu (\lambda)} \le \KL{\mu_{g (T)/n}}{\mu (\lambda)}$, and using the monotonicity of $\mu \mapsto \KL{\mu}{\mu (\lambda)}$ for $\mu \le \mu (\lambda)$, we further have that $\Bar{Y}_n^a \ge \mu_{g (T)/n}$. This argument shows that,
    \[
    \Pr_{\theta_a} (\mu (\lambda) \le U_n^a (T)) \le 
    \Pr_{\theta_a} (\mu_{g (T)/n} \le \Bar{Y}_n^a).
    \]
    Therefore,
    \[
    \sum_{n=1}^\infty \Pr_{\theta_a} (\mu (\lambda) \le U_n^a (T)) \le 
    \frac{g (T)}{\KL{\mu (\theta_a)}{\mu (\lambda)}} + 1
    + \sum_{n=n_0+1}^\infty \Pr_{\theta_a} (\mu_{g (T)/n} \le \Bar{Y}_n^a),
    \]
    where $n_0 = \left\lceil \frac{g (T)}{\KL{\mu (\theta_a)}{\mu (\lambda)}} \right\rceil$.
    
    Fix $n \ge n_0 + 1$. Then $\KL{\mu (\theta_a)}{\mu (\lambda)} > g (T)/n$, and therefore $\mu_{g (T)/n} > \mu (\theta_a)$.
    Furthermore note that $\mu_{g (T)/n}$ is increasing to $\mu (\lambda)$
    as $n$ increases, therefore $\mu_{g (T)/n}$ lives in the closed interval
    $[\mu (\theta_a), \mu (\lambda)]$, and we can apply~\autoref{lem:Chernoff} for the Markov chain that corresponds to the parameter $\theta_a$, 
    \[
        \Pr_{\theta_a} (\Bar{Y}_n^a \ge \mu_{g (T)/n}) \le 
        C_+^a e^{-n \KL{\mu_{g (T)/n}}{\mu (\theta_a)}}.
    \]
Thus we are left with the task of controlling the sum,
\[
\sum_{n = n_0 + 1}^\infty e^{-n \KL{\mu_{g (T)/n}}{\mu (\theta_a)}}.
\]
First note that by definition $\mu_{g (T)/n}$ is increasing in $n$,
therefore $\KL{\mu_{g (T)/n}}{\mu (\theta_a)}$ is positive and increasing in $n$,
hence we can perform the following integral bound,
\begin{align}
\sum_{n = n_0 + 1}^\infty e^{-n \KL{\mu_{g (T)/n}}{\mu (\theta_a)}}
&\le
\int_{\frac{g (T)}{\KL{\mu (\theta_a)}{\mu (\lambda)}}}^\infty e^{-s \KL{\mu_{g (T)/s}}{\mu (\theta_a)}} ds 
\nonumber \\
&= 
g (T) \int_0^{\KL{\mu (\theta_a)}{\mu (\lambda)}}
\frac{1}{x^2} e^{-\frac{g (T)}{x} \KL{\mu_x}{\mu (\theta_a)}}
dx.\label{eqn:integral}
\end{align}
The function $\mu \mapsto \KL{\mu}{\mu (\lambda)}$ is convex thus,
\[
\KL{\mu}{\mu (\lambda)}  \ge
\KL{\mu (\theta_a)}{\mu (\lambda)} +
\KLd{\mu (\theta_a)}{\mu (\lambda)} (\mu - \mu (\theta_a)),
\]
where $\KLd{\mu (\theta_a)}{\mu (\lambda)} = \frac{d \KL{\mu}{\mu (\lambda)}}{d \mu} \mid_{\mu = \mu (\theta_a)}$.
Plugging in $\mu = \mu_x \ge \mu (\theta_a)$,
for $x \in [0, \KL{\mu (\theta_a)}{\mu (\lambda)}]$,
we obtain
\begin{equation}\label{eqn:convexity}
\KL{\mu (\theta_a)}{\mu (\lambda)} - x 
\le 
\KLd{\mu (\theta_a)}{\mu (\lambda)}
(\mu (\theta_a) - \mu_x).
\end{equation}
From Lemma 8 in~\cite{Moulos-Ananth-19} we have that,
\begin{equation}\label{eqn:Pinkser}
\KL{\mu_x}{\mu (\theta_a)} \ge \frac{(\mu_x - \mu (\theta_a))^2}{2 \sigma_{\theta_a, \lambda}^2},
\end{equation}
where $\sigma_{\theta_a, \lambda}^2 = \sup_{\theta \in [\theta_a, \lambda]} \ddot{\Lambda} (\theta) \in (0, \infty)$.

Combining~\autoref{eqn:convexity} and~\autoref{eqn:Pinkser} we deduce that,
\[
\KL{\mu_x}{\mu (\theta_a)} \ge 
\left(
\frac{\KL{\mu (\theta_a)}{\mu (\lambda)} - x}
{\sqrt{2} \sigma_{\theta_a, \lambda} \KLd{\mu (\theta_a)}{\mu (\lambda)}}
\right)^2.
\]
Now we use this bound and break the integral in~\autoref{eqn:integral}
in two regions,
$I_1 = [0, \KL{\mu (\theta_a)}{\mu (\lambda)}/2]$
and
$I_2 = [\KL{\mu (\theta_a)}{\mu (\lambda)}/2,
\KL{\mu (\theta_a)}{\mu (\lambda)}]$.
In the first region we use the fact that $x \le \KL{\mu (\theta_a)}{\mu (\lambda)}/2$ to deduce that,
\begin{align*}
\int_{I_1}
\frac{1}{x^2} e^{-\frac{g (T)}{x} \KL{\mu_x}{\mu (\theta_a)}}
dx
&\le 
\int_{I_1}
\frac{1}{x^2}
\exp\left\{ 
-\frac{g (T)}{8 \sigma_{\theta_a, \lambda}^2 x} \left(\frac{\KL{\mu (\theta_a)}{\mu (\lambda)}}{\KLd{\mu (\theta_a)}{\mu (\lambda)}}\right)^2
\right\}
dx \\
&\le
\frac{8 \sigma_{\theta_a, \lambda}^2}{g (T)} \left(\frac{\KLd{\mu (\theta_a)}{\mu (\lambda)}}
{\KL{\mu (\theta_a)}{\mu (\lambda)}}\right)^2.
\end{align*}
In the second region we use the fact that $\KL{\mu (\theta_a)}{\mu (\lambda)}/2 \le x \le \KL{\mu (\theta_a)}{\mu (\lambda)}$ to deduce that,
\begin{align*}
\int_{I_2}
\frac{1}{x^2} e^{-\frac{g (T)}{x} \KL{\mu_x}{\mu (\theta_a)}}
dx
&\le   
\int_{I_2}
\frac{4 \exp\left\{
-\frac{(x-\KL{\mu (\theta_a)}{\mu (\lambda)})^2}
{2 \Sigma_{\theta_a,\lambda}}
\right\}}
{\KL{\mu (\theta_a)}{\mu (\lambda)}^2}
dx \\
&\le 
\int_{-\infty}^{\KL{\mu (\theta_a)}{\mu (\lambda)}}
\frac{4 \exp\left\{
-\frac{(x-\KL{\mu (\theta_a)}{\mu (\lambda)})^2}
{2 \Sigma_{\theta_a,\lambda}}
\right\}}
{\KL{\mu (\theta_a)}{\mu (\lambda)}^2}
dx \\
&= 
\frac{2 \sqrt{2 \pi \sigma_{\theta_a, \lambda}^2}}{\sqrt{g (T)}}
\sqrt{\frac{\KLd{\mu (\theta_a)}{\mu (\lambda)}^2}
{\KL{\mu (\theta_a)}{\mu (\lambda)}^3}},
\end{align*}
where $\Sigma_{\theta_a,\lambda} = \frac{\sigma_{\theta_a, \lambda}^2 \KLd{\mu (\theta_a)}{\mu (\lambda)}^2 \KL{\mu (\theta_a)}{\mu (\lambda)}}{g (T)}$.
\end{proof}

\begin{lemma}
    For every arm $a = 1, \ldots, K$,
    \begin{equation}\label{eqn:means}
    \Pr_{\theta_a} \left(\max_{n = \lceil \delta t \rceil, \ldots, t} |\Bar{Y}_n^a - \mu (\theta_a)| \ge \epsilon \right) \le \frac{c  \eta^{\delta t}}{1-\eta},
    ~\text{for}~ \delta \in (0, 1), ~ \epsilon > 0,
    \end{equation}
    where $\eta = \eta (\btheta, \epsilon) \in (0, 1)$, and $c = c (\btheta, \epsilon)$ are constants with respect to $t$.
\end{lemma}
\begin{proof}
    Using the same technique as in the proof of~\autoref{lem:Chernoff}, we have that for any $\theta \ge 0$ and any $\eta \le 0$,
    \begin{align*}
    \Pr_{\theta_a} \left(\max_{n = \lceil \delta t \rceil, \ldots, t} |\Bar{Y}_n^a - \mu (\theta_a)| \ge \epsilon \right)
    &\le \sum_{n = \lceil \delta t \rceil}^\infty
    \max_{x, y \in S}
    \frac{v_\theta^a (y)}{v_\theta^a (x)}
    e^{-n (\theta (\mu (\theta_a) + \epsilon) - \Lambda_a (\theta))} \\
    &\qquad + 
    \sum_{n = \lceil \delta t \rceil}^\infty
    \max_{x, y \in S}
    \frac{v_\eta^a (y)}{v_\eta^a (x)}
    e^{-n (\eta (\mu (\theta_a) - \epsilon) - \Lambda_a (\eta))},
    \end{align*}
    where by $\Lambda_a (\theta)$ we denote the log-Perron-Frobenious eigenvalue generated by $P_{\theta_a}$, and similarly by $v_\theta^a$ the corresponding right Perron-Frobenius eigenvector.
    
    By picking $\theta = \theta_\epsilon^a$ large enough, and $\eta = \eta_\epsilon^a$ small enough, we can ensure that
    $\theta (\mu (\theta_a) + \epsilon) - \Lambda_a (\theta) > 0$, and $\eta (\mu (\theta_a) - \epsilon) - \Lambda_a (\eta) > 0$, and so there are constants $\eta = \eta (\btheta, \epsilon) \in (0, 1)$ and
    $c = c (\btheta, \epsilon)$, such that for any $a = 1, \ldots, K$,
    \[
    \Pr_{\theta_a} \left(\max_{n = \lceil \delta t \rceil, \ldots, t} |\Bar{Y}_n^a - \mu (\theta_a)| \ge \epsilon \right) \le 
    c \sum_{n = \lceil \delta t \rceil}^\infty \eta^n \le \frac{c \eta^{\delta t}}{1 - \eta}.
    \]
\end{proof}
\section{Analysis of~\autoref{alg:KL-UCB}}\label{a:analysis}

As a proxy for the regret we will use the following quantity which involves directly the number of times each arm $a \in \{1, \ldots, N\}$ hasn't been played, and the number of times each arm $b \in \{L+1, \ldots, K\}$ has been played,
\begin{equation}\label{eqn:regret-proxy}
\Tilde{R}_\btheta^\bphi (T) =
\sum_{a=1}^N (\mu (\theta_a) - \mu (\theta_M))
\E_\btheta^\bphi [T - N_a (T)] \\
+
\sum_{b=L+1}^K (\mu (\theta_M) - \mu (\theta_b))
\E_\btheta^\bphi [N_b (T)].
\end{equation}
For the IID case $\Tilde{R}_\btheta^\bphi (T) = R_\btheta^\bphi (T)$,
and in the more general Markovian case $\Tilde{R}_\btheta^\bphi (T)$ is just a constant term apart from the expected regret $R_\btheta^\bphi (T)$.
Note that a feature that makes the case of multiple plays more delicate than the case of single plays, even for IID rewards, is the presence of the first summand in~\autoref{eqn:regret-proxy}.
For this we also need to analyze the number of times each of the best $N$
arms hasn't been played.
\begin{lemma}\label{lem:regret-proxy}
\begin{align*}
\left| R_\btheta^\bphi (T) -
\Tilde{R}_\btheta^\bphi (T)\right| \le
\sum_{a = 1}^K R_a \cdot \sum_{x \in S} |f (x)|,
\end{align*}
where $R_a = \E_{\theta_a} \left[\inf\{n \ge 1 : X_{n+1}^a = X_1^a\}\right] < \infty$.
\end{lemma}

We start the analysis by establishing the relation between the expected regret,~\autoref{eqn:regret}, and its proxy,~\autoref{eqn:regret-proxy}.
For this we will need the following lemma.

\begin{lemma}[Lemma 2.1 in~\cite{Ananth-Varaiya-Walrand-II-87}]\label{lem:renewal}
Let $\{X_n\}_{n \in \Intp{}}$ be a Markov chain on a finite state space $S$, with irreducible transition probability matrix $P$, stationary distribution $\pi$, and initial distribution $q$.
Let $\calF_n$ be the $\sigma$-field generated by $X_0, \ldots, X_n$.
Let $\tau$ be a stopping time with respect to the filtration $\{\calF_n\}_{n \in \Intp{}}$ such that $\E [\tau] < \infty$.
Define $N (x, n)$ to be the number of visits to state $x$ from time $1$ to time $n$, i.e. $N (x, n) = \sum_{k=1}^n I \{X_k = x\}$. Then
\[
\left|\E [N (x, \tau)]  - \pi (x) \E [\tau]\right|\le R,
~\text{for}~x\in S,
\]
where $R = \E [\inf \{n \ge 1 : X_{n+1} = X_1\}] < \infty$.
\end{lemma}

\begin{proof}[Proof of~\autoref{lem:regret-proxy}.] \hfill\break
First note that,
\[
S_T = \sum_{a=1}^{K} \sum_{x \in S}
f (x) N_a (x, N_a (T)).
\]
For each $a \in [K]$, using first the triangle inequality, and then~\autoref{lem:renewal} for the stopping time $N_a (T)$, we obtain,
\begin{align*}
& \left|\sum_{x \in S} f (x) (\E_\btheta^\bphi [N_a (x, N_a (T))] -
\pi_{\theta_a} (x) \E_\btheta^\bphi [N_a (T)]) \right| \\
& \qquad\le 
\sum_{x \in S} |f (x)| \left|  \E_\btheta^\bphi [N_a (x, N_a (T))] -
\pi_{\theta_a} (x) \E_\btheta^\bphi [N_a (T)] \right| \\
& \qquad\le R_a \cdot \sum_{x \in S} |f (x)|.
\end{align*}
Hence summing over $a \in [K]$, and using the triangle inequality, we see that,
\[
\left|S_T - \sum_{a=1}^K \mu (\theta_a) \Etheta [N_a (T)]
\right| \le \sum_{a=1}^K R_a \cdot \sum_{x \in S} |f (x)|.
\]
To conclude the proof note that,
\begin{align*}
& T \sum_{a=1}^M \mu (\theta_a) -
\sum_{a=1}^K \mu (\theta_a) \Etheta [N_a (T)] \\
&\qquad=
\sum_{a=1}^N \mu (\theta_a) \Etheta [T - N_a (T)] + \mu (\theta_M) (M-N)
- \mu (\theta_M) \sum_{a=N+1}^K  \Etheta [N_{a} (T)] \\
&\qquad\qquad+
\sum_{b=L+1}^K (\mu (\theta_M) - \mu (\theta_b)) \Etheta [N_b (T)] \\
&\qquad= 
\sum_{a=1}^N (\mu (\theta_a) - \mu (\theta_M)) \Etheta [T - N_a (T)] +
\sum_{b=L+1}^K (\mu (\theta_M) - \mu (\theta_b)) \Etheta [N_b (T)],
\end{align*}
where in the last equality we used the fact that $\sum_{a=1}^N \Etheta [N_a (T)] + \sum_{a=N+1}^K \Etheta [N_a (T)] = T M$.
\end{proof}

Next we show that~\autoref{alg:KL-UCB} is well-defined.
\begin{proof}[Proof of~\autoref{prop:well-defined}.] \hfill\break
Recall that $\sum_{a \in [K]} N_a (t) = t M$,
and so there exists an arm $a_1$ such that $N_{a_1} (t) \ge t M/K$.
Then $\sum_{a \in [K] - \{a_1\}} N_a (t) \ge t (M-1)$, 
and so there exists an arm $a_2 \neq a_1$ such that $N_{a_2} (t) \ge t (M-1)/(K-1)$. Inductively we can see that there exist $M$ distinct arms
$a_1, \ldots, a_M$ such that $N_{a_i} (t) \ge t (M-i+1)/(K-i+1) \ge t/K > \delta t$, for $i=1 ,\ldots, M$.
\end{proof}

\subsection{Sketch for the rest of the analysis}

Due to~\autoref{lem:regret-proxy}, it suffices to upper bound the proxy for the expected regret given in~\autoref{eqn:regret-proxy}.
Therefore, we can break the analysis in two parts:
upper bounding $\Etheta [T - N_a (T)]$, for $a = 1, \ldots, N$,
and upper bounding $\Etheta [N_b (T)]$, for $b = L+1, \ldots, K$.

For the first part, we show in~\autoref{a:analysis} that the expected number of times that an arm $a \in \{1, \ldots, N\}$
hasn't been played, is of the order of $O (\log \log T)$.
\begin{lemma}\label{lem:best-not-played}
For every arm $a = 1, \ldots, N$,
\[
\E_\btheta^\bphi [T - N_a (T)] \le 
\frac{4 e \gamma^2 N C \left\lceil\frac{2 \log \gamma}{\log \frac{1}{\delta}}\right\rceil}{\log \gamma} \log \log T
+ \gamma^{r_0} +
\frac{c \gamma^2 \eta^\delta K}{(1-\eta) (1-\eta^\delta)^3},
\]
where $\gamma, r_0, \eta, c$ and $C$ are constants with respect to $T$.
\end{lemma}

For the second part, if
$b \in \{ L+1, \ldots, K\}$, and $b \in \phi_{t+1}$, then there are three possibilities:
\begin{enumerate}
    \item $L_t \subseteq [L]$,
    and $|\Bar{Y}_a (t) - \mu (\theta_a)| \ge \epsilon$ for some $a \in L_t$,

    \item $L_t \subseteq [L]$,
    and $|\Bar{Y}_a (t) - \mu (\theta_a)| < \epsilon$ for all $a \in L_t$,
    and $b \in \phi_{t+1}$,
    
    \item $L_t \cap \{L+1, \ldots, K\} \neq \emptyset$.
    
\end{enumerate}
This means that,
\begin{align*}
\E_\btheta^\bphi [N_b (T)]
&\le M + \sum_{t=K}^{T-1} \Pr_\btheta^\bphi \left(L_t \subseteq [L], ~\text{and}~ |\Bar{Y}_a (t) - \mu (\theta_a)| \ge \epsilon ~\text{for some}~ a \in L_t\right) \\
&\quad + \sum_{t=K}^{T-1} \Pr_\btheta^\bphi \left(L_t \subseteq [L], ~\text{and}~ |\Bar{Y}_a (t) - \mu (\theta_a)| < \epsilon ~\text{for all}~ a \in L_t, ~\text{and}~ b \in \phi_{t+1}\right) \\
&\quad + \sum_{t=K}^{T-1} \Pr_\btheta^\bphi (L_t \cap \{L+1, \ldots, K\}
\neq \emptyset),
\end{align*}
and we handle each of those three terms separately.

We show that the first term is upper bounded by $O (1)$.
\begin{lemma}\label{lem:term-1}
\[
\sum_{t=K}^{T-1} \Pr_\btheta^\bphi \left(L_t \subseteq [L], ~\text{and}~ |\Bar{Y}_a (t) - \mu (\theta_a)| \ge \epsilon ~\text{for some}~ a \in L_t\right)
\le \frac{c L \eta^{\delta K}}{(1-\eta)(1-\eta^\delta)},
\]
where $c$ and $\eta$ are constant with respect to $T$.
\end{lemma}

The second term is of the order of $O (\log T)$, and it is the term that causes the overall logarithmic regret.
\begin{lemma}\label{lem:term-2}
\begin{align*}
& \sum_{t=K}^{T-1} \Pr_\btheta^\bphi \left(L_t \subseteq [L], ~\text{and}~ |\Bar{Y}_a (t) - \mu (\theta_a)| < \epsilon ~\text{for all}~ a \in L_t, ~\text{and}~ b \in \phi_{t+1}\right) \\
& \qquad\le \frac{\log T + 3 \log \log T}{\KL{\mu (\theta_b)}{\mu (\theta_M) - \epsilon}} + 1 +
8 \sigma_{\mu (\theta_a),\mu (\theta_M) - \epsilon}^2
\left(\frac{\KLd{\mu (\theta_b)}{\mu (\theta_M) - \epsilon}}{\KL{\mu (\theta_b)}{\mu (\theta_M) - \epsilon}}\right)^2 \\
& \qquad\qquad + 2 \sqrt{2 \pi \sigma_{\mu (\theta_a), \mu (\theta_M) - \epsilon}^2} \sqrt{\frac{\KLd{\mu (\theta_b)}{\mu (\theta_M) - \epsilon}^2}
{\KL{\mu (\theta_b)}{\mu (\theta_M) - \epsilon}^3}}
\left(\sqrt{\log T} + \sqrt{3 \log \log T}\right),
\end{align*}
where $\sigma_{\mu (\theta_a), \mu (\theta_M) - \epsilon}^2$, and
$\KLd{\mu (\theta_b)}{\mu (\theta_M) - \epsilon} =
\frac{d \KL{\mu}{\mu (\theta_M) - \epsilon}}{d \mu} \mid_{\mu = \mu (\theta_b)}$, are constants with respect to $T$.
\end{lemma}

Finally, we show that the third term is upper bounded by $O (\log \log T)$.
\begin{lemma}\label{lem:term-3}
\[
\sum_{t=K}^{T-1} \Pr_\btheta^\bphi (L_t \cap \{L+1, \ldots, K\} \neq \emptyset) \le
\frac{4 e \gamma^2 L C \left\lceil\frac{2 \log \gamma}{\log \frac{1}{\delta}}\right\rceil}{\log \gamma} \log \log T
+ \gamma^{r_0} +
\frac{c \gamma^2 \eta^\delta K}{(1-\eta) (1-\eta^\delta)^3},
\]
where $\gamma, r_0, \eta, c$ and $C$ are constants with respect to $T$.
\end{lemma}

This concludes the proof of~\autoref{thm:main}, modulo the four bounds of this subsection which are established in the next subsection.

\subsection{Proofs for the four bounds}

For the rest of the analysis we define the following events which describe
good behavior of the sample means and the upper confidence bounds.
For $\gamma, r \in \mathbb{Z}_{> 1}$ let,
\begin{align*}
A_r &= \bigcap_{a \in [K]} \bigcap_{\gamma^{r-1} \le t \le \gamma^{r+1}} \left\{\max_{n = \lceil \delta t \rceil, \ldots, t} |\Bar{Y}_n^a - \mu (\theta_a)| < \epsilon\right\}, \\
B_r &= \bigcap_{a \in [N]} \bigcap_{\gamma^{r-1} \le t \le \gamma^{r+1}}
\left\{\min_{n = 1, \ldots, \lceil \delta t \rceil - 1} U_n^a (t) > \mu (\theta_N)\right\}, \\
C_r &= \bigcap_{a \in [L]} \bigcap_{\gamma^{r-1} \le t \le \gamma^{r+1}}
\left\{\min_{n = 1, \ldots, \lceil \delta t \rceil - 1} U_n^a (t) > \mu (\theta_a)\right\}.
\end{align*}
Indeed, the following bounds, which rely on the concentration results of~\autoref{sec:concentration},
suggest that those events will happen with some good probability.
\begin{lemma}\label{lem:bounds}
\[
\Pr_\btheta (A_r^c) \le 
\frac{c K \eta^{\delta \gamma^{r-1}}}{(1-\eta)(1-\eta^\delta)},
\quad
\Pr_\btheta (B_r^c) \le 
\frac{4 e N C \left\lceil\frac{2 \log \gamma}{\log \frac{1}{\delta}}\right\rceil}{(r-1) \gamma^{r-1} \log \gamma},
\quad
\Pr_\btheta (C_r^c) \le 
\frac{4 e L C \left\lceil\frac{2 \log \gamma}{\log \frac{1}{\delta}}\right\rceil}{(r-1) \gamma^{r-1} \log \gamma},
\]
where $\eta \in (0, 1), ~c$ and $C$ are constants with respect to $r$.
\end{lemma}

\begin{proof}
The first bound follows directly from~\autoref{eqn:means} and a union bound.

For the second bound,
let $p = \left\lceil\frac{2 \log \gamma}{\log \frac{1}{\delta}}\right\rceil$, so that
$\left\lfloor \frac{\gamma^{r-1}}{\delta^p} \right\rfloor \ge \gamma^{r+1}$. For $i = 0, \ldots, p$ let $t_i = \left\lfloor \frac{\gamma^{r-1}}{\delta^i} \right\rfloor$, and define,
\[
D_i = \bigcap_{a \in [N]} \left\{
\min_{n=1,\ldots,t_i} U_n^a (t) > \mu (\theta_a)
\right\}.
\]
From~\autoref{eqn:maximal} we see that,
\[
\Pr_{\btheta} (D_i^c)
\le \frac{4 e N \max_{a \in [N]} C_-^a}{t_i \log t_i}
\le \frac{4 e N \max_{a \in [N]} C_-^a}{(r-1) \gamma^{r-1} \log \gamma},
\]
where $C_-^a$ is the constant from~\autoref{lem:maximal}.

Fix $a \in [N]$, and $\gamma^{r-1} \le t \le \gamma^{r+1}$.
There exists $i \in \{0, \ldots, p-1\}$ such that $t_i \le t \le t_{i+1}$,
and so $t_i > \delta t_i - 1 \ge \delta t - 1$, which gives that $t_i \ge \lceil \delta t 
\rceil - 1$.
On $D_i$, due to~\autoref{eqn:incr}, we have that,
\[
\min_{n = 1, \ldots, \lceil \delta t \rceil - 1} U_n^a (t) \ge
\min_{n = 1, \ldots, \lceil \delta t \rceil - 1} U_n^a (t_i) \ge
\min_{n=1,\ldots, t_i} U_n^a (t_i) >
\mu (\theta_a) \ge \mu (\theta_N).
\]
Therefore,
\[
\Pr_\btheta (B_r^c)
\le \sum_{i=0}^{p-1} \Pr_\btheta (D_i^c)
\le \frac{4 e N p \max_{a \in [N]} C_-^a}{(r-1) \gamma^{r-1} \log \gamma}.
\]
The third bound is established along the same lines.
\end{proof}

In order to establish~\autoref{lem:best-not-played} we need the following lemma which states that,
on $A_r \cap B_r$, an event of sufficiently large probability according to~\autoref{lem:bounds},
all the best $N$ arms are played.

\begin{lemma}[Lemma 5.3 in~\cite{Ananth-Varaiya-Walrand-I-87}]\label{lem:always-best-I}
Fix $\gamma \ge \lceil (1 - K \delta)\inv \rceil + 2$, and
let $r_0 = \lceil \log_\gamma \frac{2 K}{1 - K \delta - \gamma\inv} \rceil + 2$.
For any $r \ge r_0$,
on $A_r \cap B_r$ we have that 
$[N] \subset \phi_{t+1}$ for all
$\gamma^r \le t \le \gamma^{r+1}$.
\end{lemma}

\begin{proof}[Proof of~\autoref{lem:best-not-played}.] \hfill\break
\begin{align*}
\E_\btheta^\bphi [T - N_a (T)]
&\le \gamma^{r_0} +
\sum_{r=r_0}^{\lceil \log_\gamma (T-1) \rceil - 1} \sum_{\gamma^r \le t \le \gamma^{r+1}} \Pr_\btheta^\bphi (a \not \in \phi_{t+1}) \\
&\le \gamma^{r_0} +
\sum_{r=r_0}^{\lceil \log_\gamma (T-1) \rceil - 1} \sum_{\gamma^r \le t \le \gamma^{r+1}}
(\Pr_\btheta (A_r^c) + \Pr_\btheta (B_r^c)) \\
&\le \gamma^{r_0} + \sum_{r=r_0}^{\lceil \log_\gamma (T-1) \rceil - 1} \left(
\frac{c K \gamma^{r+1} \eta^{\delta \gamma^{r-1}}}{(1-\eta)(1-\eta^\delta)} +
\frac{4 e \gamma^2 N C \left\lceil\frac{2 \log \gamma}{\log \frac{1}{\delta}}\right\rceil}{(r-1) \log \gamma}
\right),
\end{align*}
where the second inequality follows from~\autoref{lem:always-best-I}, and the third from~\autoref{lem:bounds}.
Now we use a simple logarithmic upper bound on the harmonic number to obtain,
\[
\sum_{r=r_0}^{\lceil \log_\gamma (T-1) \rceil - 1} \frac{1}{r-1} \le 
\sum_{r=3}^{\lceil \log_\gamma (T-1) \rceil - 1} \frac{1}{r-1} \le
\log \log_\gamma T \le 
\log \log T.
\]
Finally, we can upper bound the other summand by a constant, with respect to $T$, in the following way,
\[
\sum_{r=r_0}^{\lceil \log_\gamma (T-1) \rceil - 1}
\gamma^{r-1} \eta^{\delta \gamma^{r-1}} \le 
\sum_{k=1}^\infty k \eta^{\delta k} =
\frac{\eta^\delta}{(1-\eta^\delta)^2}.
\]
\end{proof}

\begin{proof}[Proof of~\autoref{lem:term-1}.] \hfill\break
Using~\autoref{eqn:means} it is straightforward to see that
\[
\Pr_\btheta^\bphi \left(L_t \subseteq [L], ~\text{and}~ |\Bar{Y}_a (t) - \mu (\theta_a)| \ge \epsilon ~\text{for some}~ a \in L_t\right)
\le \frac{c L \eta^{\delta t}}{1-\eta},
\]
and the conclusion follows by summing the geometric series.
\end{proof}

\begin{proof}[Proof of~\autoref{lem:term-2}.] \hfill\break
Assume that $L_t \subseteq [L]$,
and $|\Bar{Y}_a (t) - \mu (\theta_a)| < \epsilon$ for all $a \in L_t$,
and $b \in \phi_{t+1}$. Then it must be the case that 
$b \equiv t + 1 \pmod{K}, ~ b \not\in L_t$, and
$U_b (t) >
\min_{a \in L_t} \Bar{Y}_a (t) >
\min_{a \in L_t} \mu (\theta_a) - \epsilon \ge
\mu (\theta_M) - \epsilon$.
This shows that,
\begin{align*}
& \Pr_\btheta^\bphi \left(L_t \subseteq [L], ~\text{and}~ |\Bar{Y}_a (t) - \mu (\theta_a)| < \epsilon ~\text{for all}~ a \in L_t, ~\text{and}~ b \in \phi_{t+1}\right) \\
& \qquad \le 
\Pr_\btheta^\bphi (b \in \phi_{t+1}, ~\text{and}~ U_b (t) > \mu (\theta_M) - \epsilon).
\end{align*}
Furthermore,
\begin{align*}
& \sum_{t=K}^{T-1} \Pr_\btheta^\bphi (b \in \phi_{t+1}, ~\text{and}~ U_b (t) > \mu (\theta_M) - \epsilon) \\
&\qquad =
\sum_{t=K}^{T-1} \sum_{n=M+1}^{M + T - K}
\Pr_\btheta^\bphi (\tau_n^b = t+1, ~\text{and}~ U_n^b (t) > \mu (\theta_M) - \epsilon) \\
&\qquad \le \sum_{t=K}^{T-1} \sum_{n=M+1}^{M + T - K}
\Pr_\btheta^\bphi (\tau_n^b = t+1, ~\text{and}~ U_n^b (T) > \mu (\theta_M) - \epsilon) \\
&\qquad = \sum_{n=M+1}^{M + T - K} \sum_{t=K}^{T-1}
\Pr_\btheta^\bphi (\tau_n^b = t+1, ~\text{and}~ U_n^b (T) > \mu (\theta_M) - \epsilon) \\
&\qquad \le \sum_{n=M+1}^{M + T - K} 
\Pr_{\theta_b} (U_n^b (T) > \mu (\theta_M) - \epsilon),
\end{align*}
where in the first inequality we used~\autoref{eqn:incr}.
Now the conclusion follows from~\autoref{eqn:sum-Chernoff}.
\end{proof}

In order to establish~\autoref{lem:term-3} we need the following lemma which states that,
on $A_r \cap C_r$, an event of sufficiently large probability according to~\autoref{lem:bounds},
only arms from $\{1, \ldots, L\}$ have been played at least $\lceil \delta t \rceil$ times and have a large sample mean.

\begin{lemma}[Lemma 5.3 B in~\cite{Ananth-Varaiya-Walrand-I-87}]\label{lem:always-best-II}
Fix $\gamma \ge \lceil (1 - K \delta)\inv \rceil + 2$, and
let $r_0 = \lceil \log_\gamma \frac{2 K}{1 - K \delta - \gamma\inv} \rceil + 2$.
For any $r \ge r_0$, on $A_r \cap C_r$ we have that 
$L_t \subseteq [L]$ for all
$\gamma^r \le t \le \gamma^{r+1}$.
\end{lemma}

\begin{proof}[Proof of~\autoref{lem:term-3}.] \hfill\break
From~\autoref{lem:always-best-II} we see that,
\[
\sum_{t=K}^{T-1} \Pr_\btheta^\bphi (L_t \cap \{L+1, \ldots, K\} \neq \emptyset)
\le \gamma^{r_0} + \sum_{r=r_0}^{\lceil \log_\gamma (T-1) \rceil - 1} \sum_{\gamma^r \le t \le \gamma^{r+1}} (\Pr_\btheta (A_r^c) + \Pr_\btheta (C_r^c)).
\]
The rest of the calculations are similar with the proof of~\autoref{lem:best-not-played}.
\end{proof}

\begin{proof}[Proof of~\autoref{cor:optimal}.] \hfill\break
In the finite-time regret bound of~\autoref{thm:main} we divide by $\log T$,
let $T$ go to $\infty$, and then let $\epsilon$ go to $0$ in order to get,
\[
\limsup_{T \to \infty} \frac{R_\btheta^\bphi (T)}{\log T} \le
\sum_{b=L+1}^K
\frac{\mu (\theta_M) -\mu (\theta_b)}
{\KL{\mu (\theta_b)}{\mu (\theta_M)}}.
\]
The conclusion now follows by using the asymptotic lower bound from~\autoref{eqn:lower-bound}.
\end{proof}

\begin{proof}[Proof of~\autoref{thm:main-IID}.] \hfill\break
The proof of~\autoref{thm:main-IID} follows along the lines the proof of~\autoref{thm:main}, by replacing instances of entries of the right Perron-Frobenius eigenvector $v_\theta (x)$ with one, and is thus omitted.
\end{proof}
\section{General Asymptotic Lower Bound}\label{a:lower-bound}

Recall from~\autoref{sec:general-family} the general one-parameter family of Markov chains $\{\Pr_\theta : \theta \in \Theta\}$, where each Markovian probability law $\Pr_\theta$ is characterized by an initial distribution $q_\theta$ and a transition probability matrix $P_\theta$. 
For this family we assume that,
\begin{gather}
P_\theta ~\text{is irreducible for all}~ \theta \in \Theta. \label{eqn:irred} \\
P_\theta (x, y) > 0 \;\Rightarrow\; P_\lambda (x, y) > 0, ~\text{for all}~ \theta, \lambda \in \Theta, ~ x, y \in S. \label{eqn:supp} \\
q_\theta (x) > 0 \;\Rightarrow\; q_\lambda (x), ~\text{for all}~ \theta, \lambda \in \Theta, ~ x \in S. \label{eqn:init}
\end{gather}
In general it is not necessary that the parameter space $\Theta$ is the whole real line, but it is assumed to satisfy the following denseness condition. For all $\lambda \in \Theta$ and all $\delta > 0$,
there exists $\lambda' \in \Theta$ such that,
\begin{equation}\label{eqn:dense}
\mu (\lambda) < \mu (\lambda') < \mu (\lambda) + \delta.
\end{equation}
Furthermore, the Kullback-Leibler divergence rate is assumed to satisfy the following continuity property.
For all $\epsilon > 0$, and for all $\theta, \lambda \in \Theta$
such that $\mu (\lambda) > \mu (\theta)$, there exists $\delta > 0$
such that,
\begin{equation}\label{eqn:cts}
\mu (\lambda) < \mu (\lambda') < \mu (\lambda) + \delta 
\;\Rightarrow\;
|\KL{\theta}{\lambda} - \KL{\theta}{\lambda'}| < \epsilon.
\end{equation}
An adaptive allocation rule $\bphi$ is called \emph{uniformly good} if,
\[
R_\btheta^\bphi (T) = o (T^\alpha),
~\text{for all}~ \btheta \in \Theta^K,
~\text{and all}~  \alpha > 0.
\]
Under those conditions~\cite{Ananth-Varaiya-Walrand-II-87} establish the following asymptotic lower bound.
\begin{theorem}[Theorem 3.1 from~\cite{Ananth-Varaiya-Walrand-II-87}]\label{thm:lower-bound}
Assume that the one-parameter family of Markov chains on the finite state space $S$, together with the reward function $f : S \to \Real{}$,
satisfy conditions~\eqref{eqn:irred},~\eqref{eqn:supp},~\eqref{eqn:init},~\eqref{eqn:dense}, and~\eqref{eqn:cts}.
Let $\bphi$ be a uniformly good allocation rule.
Fix a parameter configuration $\btheta \in \Theta^K$, and without loss of generality assume that,
\[
\mu (\theta_1) \ge \ldots \ge \mu (\theta_N) > 
\mu (\theta_{N+1}) \ldots = \mu (\theta_M) = \ldots = \mu (\theta_L) >
\mu (\theta_{L+1}) \ge \ldots \ge \mu (\theta_K).
\]
Then for every $b = L+1, \ldots, K$,
\[
\frac{1}{\KL{\theta_b}{\theta_M}}
\le 
\liminf_{T \to \infty}
\frac{\E_\btheta^\bphi \left[N_b (T)\right]}{\log T}.
\]
Consequently,
\[
\sum_{b=L+1}^K
\frac{\mu (\theta_M) -\mu (\theta_b)}
{\KL{\theta_b}{\theta_M}} \le
\liminf_{T \to \infty} \frac{R_\btheta^\bphi (T)}{\log T}.
\]
\end{theorem}

Lower bounds on the expected regret of multi-armed bandit problems
are established using a change of measure argument,
which relies on the adaptive allocation rule being uniformly good.
\cite{Lai-Robbins-85} gave the prototypical change of measure argument, for the case of i.i.d. rewards,
and \cite{Ananth-Varaiya-Walrand-II-87} extended this technique for the case of Markovian rewards.
Here we give an alternative simplified proof using the data processing inequality, an idea introduced in~\cite{Kaufmann-Cappe-Garivier-16,Combes-Proutiere-14} for the i.i.d. case.

We first set up some notation.
Denote by $\calF_T$ the $\sigma$-field generated by the random variables $\phi_1, \ldots, \phi_T,
 \{X_n^1\}_{n=0}^{N_1 (T)}, \ldots, \{X_n^K\}_{n=0}^{N_K (T)}$,
and let $\Pr_\btheta^\bphi \mid_{\calF_T}$ be the restriction of the probability distribution $\Pr_\btheta^\bphi$ on $\calF_T$.
For two probability distributions $\Pr$ and $\Qr$ over the same measurable space we define the \emph{Kullback-Leibler divergence} between $\Pr$ and $\Qr$ as
\[
\KL{\Pr}{\Qr} =
\begin{cases}
\E_{\Pr} \left[\log \frac{d \Pr}{d \Qr} \right], &\text{if}~ \Pr \ll \Qr, \\
\infty, &\text{otherwise},
\end{cases}
\]
where $\frac{d \Pr}{d \Qr}$ denotes the Radon-Nikodym derivative, when $\Pr$ is absolutely continuous with respect to $\Qr$.
Note that we have used the same notation as for the
Kullback-Leibler divergence rate between two Markov chains,
but it should be clear from the arguments whether we refer to the divergence or the divergence rate.
For $p,q \in [0, 1]$, the \emph{binary Kullback-Leibler divergence} is denoted by
\[
\KLb{p}{q} = p \log \frac{p}{q} + (1-p) \log \frac{1-p}{1-q}.
\]
The following lemma, from~\cite{Moulos19-bandits-identification}, will be crucial in establishing the lower bound.
\begin{lemma}[Lemma 1 in~\cite{Moulos19-bandits-identification}]\label{lem:KL-bound}
Let $\btheta, \blambda \in \Theta^K$ be two parameter configurations. 
Let $\tau$ be a stopping time with respect to $(\calF_t)_{t \in \Intpp{}}$, with $\Etheta [\tau], ~ \Elambda [\tau] < \infty$.
Then
\[
\KL{\Prtheta \mid_{\calF_\tau}}{\Prlambda \mid_{\calF_\tau}}
\le
\sum_{a=1}^K \KL{q_{\theta_a}}{q_{\lambda_a}} +
\sum_{a=1}^K \left(\Etheta [N_a (\tau)] + R_{\theta_a}\right) \KL{\theta_a}{\lambda_a},
\]
where $R_{\theta_a} = \E_{\theta_a}\: \left[\inf \{n \ge 1 : X_{n+1}^a = X_ 1^a\}\right] < \infty$, the first summand is finite due to~\eqref{eqn:init}, and the second summand is finite due to~\eqref{eqn:supp}.
\end{lemma}

\begin{proof}[Proof of~\autoref{thm:lower-bound}.] \hfill\break
Fix $b \in \{L+1, \ldots, K\}$,
and $\epsilon > 0$.
Due to~\autoref{eqn:dense} and~\autoref{eqn:cts},
there exists $\lambda \in \Theta$ such that
\[
\mu (\theta_M) < \mu (\lambda) , ~\text{and}~
|\KL{\theta_b}{\theta_M} - \KL{\theta_b}{\lambda}| < \epsilon.
\]
We consider the parameter configuration
$\blambda = (\lambda_1, \ldots, \lambda_K)$ given by,
\[
\lambda_a =
\begin{cases}
\theta_a, & \text{if}~ a \neq b, \\
\lambda, & \text{if}~ a = b.
\end{cases}
\]
Using~\autoref{lem:KL-bound} we obtain,
\[
\KL{\Pr_\btheta^\bphi \mid_{\calF_T}}{\Pr_\blambda^\bphi \mid_{\calF_T}}
\le
\KL{q_{\theta_b}}{q_\lambda} + R_{\theta_b} \KL{\theta_b}{\lambda} +
\E_\btheta^\bphi \left[N_b (T)\right] \KL{\theta_b}{\lambda}.
\]
From the data processing inequality, see the book of~\cite{CovThom06}, we have that for any 
event $\calE\in\calF_T$,
\[
\KLb{\Pr_\btheta^\bphi (\calE)}{\Pr_\blambda^\bphi (\calE)} \le
\KL{\Pr_\btheta^\bphi \mid_{\calF_T}}{\Pr_\blambda^\bphi \mid_{\calF_T}}.
\]
We select $\calE = \{N_b (T) \ge \sqrt{T}\}$.
Then using Markov's inequality, and the fact that $\bphi$
is uniformly good we obtain for any $\alpha > 0$,
\[
\Pr_\btheta^\bphi (\calE)
\le \frac{\E_\btheta^\bphi [N_b (T)]}{\sqrt{T}}
= \frac{o (T^\alpha)}{\sqrt{T}}, \quad
\Pr_\blambda^\bphi (\calE^c)
\le \frac{\E_\blambda^\bphi [T - N_b (T)]}{T-\sqrt{T}}
= \frac{o (T^\alpha)}{T-\sqrt{T}}.
\]
Using those two inequalities we see that,
\[
\liminf_{T \to \infty}
\frac{\KLb{\Pr_\btheta^\bphi (\calE)}
{\Pr_\blambda^\bphi (\calE)}}
{\log T}
=
\liminf_{T \to \infty}
\frac{\log \frac{1}{\Prlambda (\calE^c)}}
{\log T}
\ge 
\lim_{T \to \infty}
\frac{\log \frac{T-\sqrt{T}}{o (T^\alpha)}}{\log T} = 1.
\]
Therefore,
\[
\liminf_{T \to \infty} \frac{\E_\btheta^\bphi [N^b (T)]}{\log T} \ge 
\frac{1}{\KL{\theta_b}{\lambda}} \ge 
\frac{1}{\KL{\theta_b}{\theta_M} + \epsilon},
\]
and the first part of~\autoref{thm:lower-bound} follows by letting $\epsilon$ go to $0$. The second part follows from~\autoref{lem:regret-proxy}, and~\autoref{eqn:regret-proxy}.
\end{proof}
\end{appendices}

\end{document}